\documentclass{birkjour_t2}
 \theoremstyle{definition}
 
 \theoremstyle{remark}

 \numberwithin{equation}{section}

\usepackage{graphicx}
\usepackage{amssymb, amsmath, amsthm}
\usepackage{amsfonts}
\usepackage{amssymb}
\usepackage{float}
\usepackage{mathrsfs}
\usepackage{enumitem}
\newtheorem{theorem}{Theorem}

\newtheorem{definition}[theorem]{Definition}

\newtheorem{lemma}[theorem]{Lemma}

\newtheorem{proposition}[theorem]{Proposition}
\newtheorem{remark}[theorem]{Remark}

\newcommand{\dis}{\displaystyle}

\newcommand{\Z}{\mathbb Z}

\newcommand{\R}{\mathbb R}

\renewcommand{\theequation}{\thesection.\arabic{equation}}
\numberwithin{equation}{section}
\numberwithin{theorem}{section}

\numberwithin{figure}{section}

\begin{document}

%
%
%
%
%
%
%
%
%

\title[Magnetogeostrophic equations]
 {Solutions to a class of forced drift-diffusion equations with applications to the magneto-geostrophic equations}

\author{Susan Friedlander}

\address{Department of Mathematics\\
University of Southern California}

\email{susanfri@usc.edu}

\author{Anthony Suen}

\address{Department of Mathematics and Information Technology\\
The Education University of Hong Kong}

\email{acksuen@eduhk.hk}

\date{}



\begin{abstract}
We prove the global existence of classical solutions to a class of forced drift-diffusion equations with $L^2$ initial data and divergence free drift velocity $\{u^\nu\}_{\nu_\ge0}\subset L^\infty_t BMO^{-1}_x$, and we obtain strong convergence of solutions as the viscosity $\nu$ vanishes. We then apply our results to a family of active scalar equations which includes the three dimensional magneto-geostrophic $\{$MG$^\nu\}_{\nu\ge0}$ equation that has been proposed by Moffatt in the context of magnetostrophic turbulence in the Earth's fluid core. We prove the existence of a compact global attractor $\{\mathcal{A}^\nu\}_{\nu\ge0}$ in $L^2(\mathbb{T}^3)$ for the MG$^\nu$ equations including the critical equation where $\nu=0$. Furthermore, we obtain the upper semicontinuity of the global attractor as $\nu$ vanishes.
\end{abstract}

\maketitle
\section{Introduction}\label{Introduction}

Our motivation for addressing the limiting behaviour of a class of drift diffusion equations comes from a model proposed by Moffatt and Loper \cite{ML94}, Moffatt \cite{M08} for magnetostrophic turbulence in the Earth's fluid core. This model is derived from the full three dimensional magnetohydrodynamic equations (MHD) in the context of a rapidly rotating, densely stratified, electrically conducting fluid. For discussions about the MHD equations in geophysical contexts see, for example, \cite{A15}, \cite{C54}, \cite{D01}, \cite{KA15}, \cite{M78}, \cite{RK13}. Following the notation of Moffatt and Loper \cite{ML94}, we write the equations in terms of dimensionless variables. The orders of magnitude of the resulting non-dimensional parameters are motivated by the physical postulates of their model:
\begin{align}
R_0(\partial_t u+u\cdot\nabla u)+e_3\times u&=-\nabla P+e_2\cdot\nabla b+R_m b\cdot\nabla b+\theta e_3+\nu \Delta u,\label{1.1-1}\\
R_m[\partial_t b+u\cdot\nabla b-b\cdot\nabla u]&=e_2\cdot\nabla u+\Delta b,\label{1.1-2}\\
\partial_t\theta+u\cdot\nabla\theta&=\kappa\Delta\theta+S,\label{1.1-3}\\
\nabla\cdot u&=0,\nabla\cdot b=0\label{1.1-4}.
\end{align}
The unknowns are $u(t,x)$ the velocity, $b(t,x)$ the magnetic field (both vector valued) and $\theta(t,x)$ the scalar (temperature field of the fluid). $P$ is the sum of the fluid and magnetic pressures, and the Cartesian unit vectors are given by $e_1, e_2$ and $e_3$. The physical forces governing this system are the Coriolis force, the Lorentz force and gravity acting via buoyancy, while the equation for the temperature is driven by a smooth function $S(x)$ that represents the external forcing of the MHD system.

The non-dimensional parameters in \eqref{1.1-1}-\eqref{1.1-4} are $R_0$ the Rossby number, $R_m$ the magnetic Reynolds number, $\nu$ a (non-dimensional) viscosity and $\kappa$ a (non-dimensional) thermal diffusivity. Moffatt and Loper argue that for the geophysical context they are modelling, all these parameters are small, with $\nu$ and $\kappa$ being extremely small. We note that the ratio of the Coriolis to Lorentz forces in their model is of order 1, so for notational simplicity we have set this parameter, denoted by $N^2$ in \cite{ML94}, equal to 1. Hence in \eqref{1.1-1}-\eqref{1.1-4} we have a system derived from an important physical problem that is very rich in small parameters. The mathematical properties of this system under various settings of some of the parameters to zero, or in the vanishing limits have been addressed in a sequence of different articles, \cite{FFGR17}, \cite{FRV12}, \cite{FRV14}, \cite{FS15}, \cite{FV11a}, \cite{FV11b}, \cite{FV12}. Although the physically relevant boundary for a model of the Earth's fluid core is a spherical annulus, for mathematical tractability these studies have considered the system on the periodic domain $\mathbb{T}^3$, with all fields being mean free, a condition which is preserved by the equations. In this present paper we study the forced system \eqref{1.1-1}-\eqref{1.1-4} under these boundary conditions.

The system as investigated by Moffatt and Loper neglects the terms multiplied by $R_0$ and $R_m$ in comparison with the remaining terms. Essentially this means that the evolution equations \eqref{1.1-1} and \eqref{1.1-2} for the coupled velocity and magnetic vectors take a simplified ``quasi-static'' form. A linear relationship is then established between the vector fields and the scalar temperature $\theta$. The sole remaining nonlinearity in the system occurs in the evolution equation for $\theta$ given by \eqref{1.1-3}. This equation is then an advection-diffusion equation where the constitutive law that relates the divergence free velocity vector $u$ and the scalar $\theta$ is obtained from the reduced linear system
\begin{align}
e_3\times u&=-\nabla P+e_2\cdot\nabla b+\theta e_3+\nu\Delta u,\label{1.1-5}\\
0&=e_2\cdot\nabla u+\Delta b,\label{1.1-6}\\
\nabla\cdot u&=0, \nabla\cdot b=0.\label{1.1-7}
\end{align} 
This system encodes the vestiges of the physics in the problem, namely the Coriolis force, the Lorentz force and gravity. Vector manipulations of \eqref{1.1-5}-\eqref{1.1-7} give the expression
\begin{align}\label{1.1-8}
\left\{[\nu\Delta^2-(e_2\cdot\nabla)^2]^2+(e_3\cdot\nabla)^2\Delta\right\}u=-[\nu\Delta^2-(e_2\cdot\nabla)^2]&\nabla\times(e_3\times\nabla\theta)\notag\\
\qquad&+(e_3\cdot\nabla)\Delta(e_3\times\nabla\theta).
\end{align}
We study the forced active scalar equation
\begin{align}
\label{1.1-9} \left\{ \begin{array}{l}
\partial_t\theta+u\cdot\nabla\theta=\kappa\Delta\theta+S, \\
u=M^\nu[\theta],\theta(0,x)=\theta_0(x)
\end{array}\right.
\end{align}
via an examination of the Fourier multiplier symbol of the operator $M^\nu$ obtained from \eqref{1.1-8}. This active scalar equation is called the magnetogeostrophic (MG) equation. We also refer to \eqref{1.1-9} as the MG$^\nu$ equation when $\nu>0$, and to the case when $\nu=0$ as the MG$^0$ equation. In Section~\ref{The MG Equations}, we write the explicit expression for the Fourier multiplier symbol of $M^\nu$ obtained from \eqref{1.1-8}. We observed that the limit $\nu\rightarrow0$ is a highly singular limit: in particular, when $\nu>0$ the operator $M^\nu$ is smoothing of degree 2, however when $\nu=0$ the operator $M^0$ is singular of degree $-1$. The goals of the current article are to examine the convergence of solutions to \eqref{1.1-9} in the limit as the viscosity $\nu$ goes to zero and to study the long time behaviour of the forced system. 

Friedlander and Vicol \cite{FV11a} analyzed the unforced $S=0$ system \eqref{1.1-8}-\eqref{1.1-9} with the viscosity parameter $\nu$ set to zero, i.e. the unforced MG$^0$ equation. In this situation the drift diffusion equation \eqref{1.1-9} is critical in the sense of the derivative balance between the advection and the diffusion term. They used De Giorgi techniques to obtain global well-posedness results for the unforced critical MG$^0$ equation in a similar manner to the proof of global well-possedness given by Caffarelli and Vaseur \cite{CV10} for the critical SQG equation. In Section~\ref{Existence of smooth solutions} of this present paper
we verify that the technical details of the De Giorgi techniques are, in fact, valid for drift diffusion equations with a smooth force. This procedure leads to the proof of Theorem~\ref{Existence of smooth solutions Thm}, namely the existence of smooth classical solutions to the forced MG$^\nu$, $\nu\ge0$ equations. The existence of uniform $H^s$ bounds on these smooth solutions proved in Section~\ref{Uniform bounds on smooth solutions and proof of Theorem 1.2} implies convergence of solutions as $\nu$ vanishes which is stated in Theorem~\ref{Convergence of solutions as nu goes to 0}.

The second main result of this current paper concerns the existence of a global attractor for the critical MG$^0$ equation. The issue of the existence of a global attractor for active scalar equations is important in the general context of the long time averages of solutions to forced fluid equations (c.f., \cite{CF89}). In particular there are recent results concerning the existence of a global attractor for the dynamics of the forced critical SQG equation \cite{CD14}, \cite{CCV16}, \cite{CTV14a}, \cite{CTV14b}. In \cite{CD14}, Cheskidov and Dai prove that the forced critical SQG equation possesses a global attractor in $L^2(\mathbb{T}^3)$ provided the force $S$ is in $L^p(\mathbb{T}^3)$, $p>2$. They use ``classical'' viscosity solutions and the abstract framework of evolutionary systems introduced by Cheskidov and Foias \cite{C09}, \cite{CF06}. We prove an analogous result for the forced three dimensional critical MG$^0$ equation using the concept of a ``{\it vanishing viscosity}''{\it solution} that arises naturally from the results of Theorem~\ref{Convergence of solutions as nu goes to 0} concerning the convergence in the ``vanishing viscosity'' limit of solutions for the MG$^\nu$ and MG$^0$ equations. Our treatment of the MG$^\nu$, $\nu\ge0$, equations is novel in the sense that traditionally the existence of a global attractor is obtained from asymptotic compactness. For the MG equations the asymptotic compactness is not known a priori and the existence of the global attractor follows from
\begin{itemize}
\item[(a)] the energy equality, which implies that the energy cannot grow rapidly (see Proposition~\ref{energy equality proposition}),
\item[(b)] the absence of anomalous dissipation for complete bounded trajectories.
\end{itemize}
We note that the suggestion that this strategy might work for various equations was first proposed in \cite{CD14}.

The analysis in \cite{FV11a} of the unforced MG equations was given in the context of a class of drift diffusion equations where the divergence free drift velocity lies in $L_t^\infty BMO^{-1}$, which class includes the MG$^0$ equation. We follow their approach using De Giorgi techniques to obtain global well-posedness results for a class of forced drift diffusion equations. More specifically, we study the following active scalar equation in $\mathbb{T}^d\times(0,\infty)$ with $d\ge2$:
\begin{align}
\label{0.1} \left\{ \begin{array}{l}
\partial_t\theta^{\nu}+u^{\nu}\cdot\nabla\theta^{\nu}=\kappa\Delta\theta^{\nu}+S, \\
u_j^{\nu}=\partial_{i} T_{ij}^{\nu}[\theta^{\nu}],\theta(0,x)=\theta_0(x)
\end{array}\right.
\end{align}
where $\nu\ge0$. Here $\kappa>0$ is a diffusive constant, $\theta_0$ is the initial condition and $S=S(x)$ is a given smooth function that represents the forcing of the system. $\{T_{ij}^{\nu}\}_{\nu\ge0}$ is a sequence of operators which satisfy: 
\begin{enumerate}[label={\upshape{(1.\arabic*)}}, ref={\upshape{2.}\arabic*}, topsep=0.3cm, itemsep=0.15cm] 
\setcounter{enumi}{10}
\item $\partial_i\partial_j T^{\nu}_{ij}f=0$ for any smooth functions $f$.
\item $T_{ij}^\nu:L^\infty(\mathbb{T}^d)\rightarrow BMO(\mathbb{T}^d)$ are bounded for all $\nu\ge0$.
\item There exists a constant $C_{*}>0$ independent of $\nu$, such that for all $1\le i,j\le d$, $$\sup_{\nu\in(0,1]}\sup_{\{k\in\mathbb{Z}^d\}}|\widehat T^{\nu}_{ij}(k)|\le C_{*};$$
$$\sup_{\{k\in\mathbb{Z}^d\}}|\widehat T^0_{ij}(k)|\le C_{*},$$
where $T^0_{ij}=T_{ij}^{\nu}\Big|_{\nu=0}$.
\item For each $1\le i,j\le d$, $$\lim_{\nu\rightarrow0}\sum_{k\in\mathbb{Z}^d}|\widehat{T^{\nu}_{ij}}(k)-\widehat{T^0_{ij}}(k)|^2|\widehat{g}(k)|^2=0$$
for all $g\in L^2$.
\end{enumerate}

The main results that we prove for the forced problem are stated in the following theorems:

\begin{theorem}[Existence of smooth solutions]\label{Existence of smooth solutions Thm}
Let $\theta_0\in L^2$, $S\in C^{\infty}$ and $\kappa>0$ be given, and assume that $\{T_{ij}^{\nu}\}_{\nu\ge0}$ satisfy conditions (1.11)-(1.14). There exists a classical solution $\theta^{\nu}(t,x)\in C^{\infty}((0,\infty)\times \mathbb{T}^d)$ of \eqref{0.1}, evolving from $\theta_0$ for all $\nu\ge0$.
\end{theorem}

\begin{theorem}[Convergence of solutions as $\nu\rightarrow0$]\label{Convergence of solutions as nu goes to 0}
Let $\theta_0\in L^2$, $S\in C^{\infty}$ and $\kappa>0$ be given, and assume that $\{T_{ij}^{\nu}\}_{\nu\ge0}$ satisfy conditions (1.11)-(1.14). If $\theta^{\nu},\theta$ are $C^\infty$ smooth classical solutions of the system \eqref{0.1} for $\nu>0$ and $\nu=0$ respectively with initial data $\theta_0$, then given $\tau>0$, for all $s\ge0$, we have
\setcounter{equation}{14}
\begin{align}\label{1.6-}
\lim_{\nu\rightarrow0} \|(\theta^{\nu}-\theta)(t,\cdot)\|_{ H^s}=0,
\end{align}
whenever $t\ge\tau$.
\end{theorem}

\begin{theorem}[Existence of a global attractor for the MG equation]\label{Existence of a global attractor for the MG equation}
Let $\theta_0\in L^2$, $S\in C^{\infty}$ and $\kappa>0$ be given. The system \eqref{1.1-9} with $\nu=0$ possesses a compact global attractor $\mathcal{A}$ in $L^2(\mathbb{T}^3)$, namely
\begin{align*}
\mbox{$\mathcal{A}=\{\theta_0:\theta_0=\theta(0)$ for some bounded complete ``vanishing viscosity'' solution $\theta(t)\}$.}
\end{align*}
For any bounded set $\mathcal{B}\subset L^2(\mathbb{T}^3)$, and for any $\varepsilon,T>0$, there exists $t_0$ such that for any $t_1>t_0$, every ``vanishing viscosity'' solution $\theta(t)$ with $\theta(0)\in\mathcal{B}$ satisfies 
\begin{align*}
\|\theta(t)-x(t)\|_{L^2}<\varepsilon, \forall t\in[t_1,t_1+T],
\end{align*}
for some complete trajectory $x(t)$ on the global attractor $(x(t)\in\mathcal{A}, \forall t\in(-\infty,\infty))$.

Furthermore, for $\nu\in[0,1]$, there exists a compact global attractor $\mathcal{A}^{\nu}\subset L^2$ for \eqref{1.1-9} such that $\mathcal{A}^{0}=\mathcal{A}$ and $\mathcal{A}^{\nu}$ is {\it upper semicontinuous} at $\nu=0$, which means that
\begin{align}\label{uc}
\mbox{$\sup_{\phi\in\mathcal{A}^{\nu}}\inf_{\psi\in\mathcal{A}}\|\phi-\psi\|_{L^2}\rightarrow0$ as $\nu\rightarrow0$.}
\end{align}
\end{theorem}

Our paper is organised as follows. In Section~\ref{Preliminaries and notations}, we give some preliminaries and notations which will be used in later sections. In Section~\ref{Existence of smooth solutions}, we state and prove Theorem~\ref{Existence of smooth solutions Thm}, namely the existence of a smooth solution to \eqref{0.1}. In Section~4, we obtain a uniform $H^s$-bound on smooth solutions to \eqref{0.1} and prove Theorem~\ref{Convergence of solutions as nu goes to 0}. In Section~\ref{The MG Equations} we show that the MG$^\nu$ equation with $\nu\ge0$ satisfies the general conditions formulated for the active scalar equation in (1.10)-(1.14). Hence Theorem~\ref{Convergence of solutions as nu goes to 0} can be applied to prove convergence as $\nu\rightarrow0$ of solutions of the subcritical MG$^\nu$ equation to solutions of the critical MG$^0$ equation. In Section~\ref{The Existence of a Global Attractor}, we introduce the concept of a ``vanishing viscosity'' weak solution of the MG$^0$ equation and we prove that the forced critical MG$^0$ equation possesses a compact global attractor in $L^2(\mathbb{T}^3)$ satisfying \eqref{uc}. 

\section{Preliminaries and notations}\label{Preliminaries and notations}

In this section, we give some preliminaries and notations which are useful in later sections (also refer to \cite{D05} for more detailed discussion).

\medskip 

\noindent Throughout this paper, we shall denote $L^p([0,\infty);L^q(\mathbb{T}^d))$ by $L^p_t L^q_x$ for $1\le p,q\le\infty$, and similarly for $L^p_t BMO_{x}$, etc. Also, $L^p_{t,x}(I\times A) = L^p(I;L^p(A))$ for any $I\subset\R$ and $A\subset\mathbb{T}^d$.

\medskip

\noindent Let $\varphi$ be a smooth function valued in $[0,1]$ such that $\varphi$ is supported in the shell $\{\xi\in\R^d:\frac{1}{2}\le|\xi|\le4\}$. Denote 
\begin{align*}
h_j(x)=\sum_{k\in\tilde{\mathbb{Z}}^d}\varphi(2^{-j}k)e^{ik\cdot x},\,x\in\mathbb{T}^d,
\end{align*}
where $\tilde{\mathbb{Z}}^d$ is the dual lattice associated to $\mathbb{T}^d$. For $j\in\mathbb{Z}$, we define the periodic dyadic blocks as follows:
\begin{align*}
\Delta_jf(x)=\frac{1}{(2\pi)^d}\int_{\mathbb{T}^d}h_j(y)f(x-y)dy,\,x\in\mathbb{T}^d.
\end{align*}
For $s\in\R$ and $1\le p,q\le\infty$, the Besov norm for $B^s_{p,q}=B^s_{p,q}(\mathbb{T}^d)$ is defined as (summation over repeated indices is understood): 
\begin{align*}
\|f\|_{B^s_{p,q}}=\|2^{js}\|\Delta_j f\|_{L^p}\|_{\ell^q(\mathbb{Z})}.
\end{align*}

\noindent We also recall the Chemin-Lerner space-time Besov space $\tilde{L}^r(I; B^s_{p,q})$, with norm given by
\begin{align*}
\|f\|_{\tilde{L}^r(I; B^s_{p,q})}=\left\|2^{js}\left(\int_{I}\|\Delta_j f(t,\cdot)\|^r_{L^p} dt\right)^\frac{1}{r}\right\|_{\ell^q(\mathbb{Z})},
\end{align*}
where $s\in\R$, $1\le r,p,q\le\infty$ and $I$ is a time interval.

\medskip

\noindent We will make use of the following well-known embedding theorems in later sections (refer to \cite{BCD11} for more details).

\medskip

\noindent{\it Gagliardo-Nirenberg-Sobolev inequality:} Assume that $1\le p\le d$. There exists a constant $C>0$ depending only on $p$ and $d$, such that
\begin{align}
\|f\|_{L^{dp/(d-p)}}\le C\|Df\|_{L^p}\label{2a},
\end{align}
for all $f\in C^1_0(\mathbb{T}^d)$.

\medskip

\noindent{\it Gagliardo-Nirenberg interpolation inequality:} Fix $1\le q,r\le\infty$ and $m\in\mathbb{N}$. Suppose that $m,p,\gamma,j$ satisfy
\begin{align*}
\frac{1}{p}=\frac{j}{d}+(\frac{1}{r}-\frac{m}{d})\gamma+\frac{1-\gamma}{q},\qquad\frac{j}{m}\le\gamma\le1,
\end{align*}
where $d$ is the dimension, then
\begin{align}
\|D^j f\|_{L^p}\le C\|D^m f\|^\gamma_{L^r}\|f\|^{1-\gamma}_{L^q},\label{2b}
\end{align} 
where $C>0$ is a constant which depends only on $m, d, j, q, r,\gamma$.

\medskip

\noindent{\it Gagliardo-Nirenberg interpolation inequality for Sobolev space:} Let $q,r\in(1,\infty]$ and $\sigma,s\in(0,\infty)$ with $\sigma<s$. There exists a positive dimensional constant $C$ such that
\begin{align}
\|f\|_{ W^{\sigma,p}}\le C\|f\|^\gamma_{L^q}\|f\|^{1-\gamma}_{ W^{s,r}},\label{2c}
\end{align}
with $\frac{1}{p}=\frac{\gamma}{q}+\frac{1-\gamma}{r}$ and $\gamma=1-\frac{\sigma}{s}$.

\medskip

\noindent{\it Besov Embedding Theorem:} Let $s\in\R$. If $1\le p_1\le p_2\le\infty$ and $1\le r_1\le r_2\le\infty$, then
\begin{align}
 B^s_{p_1,r_1}\hookrightarrow B^{s-d(\frac{1}{p_1}-\frac{1}{p_2})}_{p_2,r_2}.\label{2d}
\end{align}

\section{Existence of smooth solutions}\label{Existence of smooth solutions}

In this section, we prove the existence of smooth solutions to the forced non-linear problem \eqref{0.1}. It can be stated as follows.

\begin{theorem}\label{existence of smooth solutions sub Thm}
Let $\theta_0\in L^2$, $S\in C^{\infty}$ and $\kappa>0$ be given, and assume that $\{T_{ij}^{\nu}\}_{\nu\ge0}$ satisfy conditions (1.11)-(1.14). There exists a classical solution $\theta^{\nu}(t,x)\in C^{\infty}((0,\infty)\times \mathbb{T}^d)$ of \eqref{0.1}, evolving from $\theta_0$ for all $\nu\ge0$.
\end{theorem}

\noindent{\bf The Linear Problem}

\medskip

\noindent Theorem~\ref{existence of smooth solutions sub Thm} can be proved by the similar method as given in \cite{FV11a} with modification for the presence of a forcing term $S$ in \eqref{0.1}. Following the proof given in \cite{FV11a}, we first consider the {\it linear problem}:
\begin{align}
\partial_t\theta +(v\cdot\nabla)\theta=\kappa\Delta\theta+S,\label{2.0.1}
\end{align}
where the velocity vector $v(t,x)=(v_1(t,x),\cdots,v_d(t,x))\in L^2((0,\infty)\times\mathbb{T}^d)$ is given, and $(t,x)\in[0,\infty)\times\mathbb{T}^d$. Additionally, let $v$ satisfies
\begin{align}
\partial_jv_j(t,x)=0\label{2.0.2}
\end{align}
in the sense of distributions. We express $v_j$ as
\begin{align}
\mbox{$v_j(t,x)=\partial_j\partial_iV_{ij}(t,x)$ in $[0,\infty)\times \mathbb{T}^d$,}\label{2.0.3}
\end{align}
and we denoted $V_{ij}=-(-\Delta)^{-1}\partial_i v_j$. The matrix $\{V_{ij}\}^d_{i,j =1}$ is given, and satisfies
\begin{align}
V_{ij}\in L^{\infty}((0,\infty);L^2(\mathbb{T}^d)\cap L^2((0,\infty);H^1(\mathbb{T}^d))\label{2.0.4}
\end{align}
for all $i,j\in\{1,...,d\}$.

\medskip

\noindent We first prove the following proposition for the existence of smooth solutions to \eqref{2.0.1}-\eqref{2.0.4}. 
\begin{proposition}
Given $\theta_0\in L^2$ and $S\in C^{\infty}$, and assume that $\{V_{ij}\}$ satisfies \eqref{2.0.4}. Let 
\begin{align}\label{2.0.2-}
\theta\in L^\infty([0,\infty); L^2(\mathbb{T}^d))\cap L^2((0,\infty);H^1(\mathbb{T}^d))
\end{align}
be a global weak solution of the initial value problem associated to \eqref{2.0.1}-\eqref{2.0.4}. If additionally we have $V_{ij}\in L^\infty([t_0,\infty);BMO(\mathbb{T}^d))$ for all $i,j\in\{1,...,d\}$ and some $t_0>0$, then there exists $\alpha>0$ such that $\theta\in C^\alpha([t_0,\infty)\times\mathbb{T}^d)$.
\end{proposition}

\begin{remark}
Note that for divergence-free $v\in L^2_{t,x}$, the existence of a weak solution $\theta$ to \eqref{2.0.1}-\eqref{2.0.4} evolving from $\theta_0\in L^2$ is well-known (for example, see \cite{S06} where the more general case $v\in L^1_{loc}$ is discussed, also \cite{CV10} and references therein). Here $\theta$ is a weak solution to \eqref{2.0.1}-\eqref{2.0.4} in the sense that $\theta$ satisfies \eqref{2.0.1}-\eqref{2.0.4} in a distributional sense, that is, for any $\phi\in C_0^\infty((0,T)\times\mathbb{T}^3)$,
\begin{align*}
-\int_0^T\langle\theta,\phi_t\rangle dt-\int_0^T\langle v\theta, \nabla\phi\rangle dt+\int_0^T\langle\nabla\theta,\nabla\phi\rangle dt=\langle\theta_0,\phi(0,x)\rangle+\int_0^T\langle S,\phi\rangle dt,
\end{align*}
where $\langle\cdot,\cdot\rangle$ is the standard $L^2$-inner product on $\mathbb{T}^d$. 
\end{remark}

\begin{proof}[Proof of Proposition~3.2] In view of Theorem~2.1 in \cite{FV11a}, we prove Proposition~3.2 in the following steps. Throughout the proof, we assume $\kappa\equiv1$ for simplicity. 
\medskip

\noindent {\bf Step 1:} A weak solution to \eqref{2.0.1}-\eqref{2.0.4} is bounded for positive time. In other words, there exists a positive constant $C$ such that for all $t > 0$,
\begin{align}\label{2.0.5}
\|\theta(t,\cdot)\|_{L^\infty}\le C\|S\|_{L^\infty}^{-\frac{d}{d+4}}(\|\theta_0\|_{L^2}+\|S\|_{L^\infty})(1+t^{-\frac{d}{2}}).
\end{align}
\begin{proof}
It follows by the similar method in proving Lemma~2.3 as in \cite{CD14}. First, for $t>0$, we have the energy inequalities
\begin{align}\label{energy inequality_1}
\|\theta(t,\cdot)\|_{L^2}\le C\|\theta_0\|_{L^2}+t\|S\|_{L^\infty},
\end{align}
\begin{align}\label{energy inequality_2}
\int_0^t\int|\nabla\theta(s,x)|^2dxds\le C\left[\|\theta_0\|_{L^2}^2+t\|S\|_{L^\infty}^2\right],
\end{align}
And for $h>0$, we have the following level set energy inequality for the truncated function $(\theta-h)_{+}$:
\begin{align}\label{2.0.5-2}
\int|(\theta(t_2,x)-h)_{+}|^2dx+2\int_{t_1}^{t_2}\!\!\!\int|\nabla(\theta-h)|^2\le\int|(\theta(t_1,x)-h)_{+}|^2dx+2\int_{t_1}^{t_2}\!\!\!\int|S(\theta-h)_{+}|
\end{align}
for all $0<t_1<t_2<\infty$. 

Next, we apply De Giorgi iteration method based on \eqref{2.0.5-2}. First we fix $t_0>0$ and define 
\begin{align*}
c_n=\sup_{t_n\le t\le t_0}\int|\theta_n|^2+2\int_{t_n}^{\infty}\!\!\!\int|\nabla\theta_n|^2,
\end{align*}
where $\theta_n=(\theta(t,\cdot)-h_n)_{+}$, $t_n=t_0-\frac{t_0}{2^n}$, $h_n=H-\frac{H}{2^n}$ and $H$ to be chosen later. Then we have
\begin{align}\label{2.0.5-3}
c_n&\le\frac{2^{n+1}}{t_0}\int_{t_{n-1}}^\infty\int\theta_n^2\cdot\chi_{\{\theta_n>0\}}+2\int_{t_{n-1}}^\infty\int|S\theta_n|,
\end{align}
where $\chi$ is the characteristic function and
\begin{align*}
\chi_{\{\theta_n>0\}}\le\frac{2^n}{H}\theta_{n-1}.
\end{align*} 
For the first term on the right side of \eqref{2.0.5-3}, using Gagliardo-Nirenberg inequality \eqref{2b}, it can be estimated as follows.
\begin{align*}
\frac{2^{n+1}}{t_0}\int_{t_{n-1}}^\infty\int\theta_n^2\cdot\chi_{\{\theta_n>0\}}&\le\frac{2^{n+1}}{t_0}\left(\frac{2^n}{H}\right)^{p-2}\int_{t_{n-1}}^\infty\int\theta_{n-1}^p\notag\\
&\le C\frac{2^{n(p-1)+1}}{t_0H^{p-2}}\int_{t_{n-1}}^\infty\left(\int|\theta_{n-1}|^2\right)^{\frac{p(1-\gamma)}{2}}\left(\int|\nabla\theta_{n-1}|^2\right)^{\frac{p\gamma}{2}}\notag\\
&\le \frac{C}{t_0}\cdot\frac{2^{n(p-1)+1}}{H^{p-2}}C_{n-1}^\frac{p}{2},
\end{align*}
where $\gamma=\frac{2}{p}$, $p=2(1+\frac{2}{d})$ and $C>0$ is a dimensional constant independent of $n$. Similarly, the second term on the right side of \eqref{2.0.5-3} is bounded by
\begin{align*}
\int_{t_{n-1}}^\infty\int|S\theta_n|&\le\|S\|_{L^\infty}\int_{t_{n-1}}^\infty\int|\theta_n|\cdot\chi^{p-1}_{\{\theta_n>0\}}\\
&\le\|S\|_{L^\infty}\left(\frac{2^n}{H}\right)^{p-1}\int_{t_{n-1}}^\infty\int|\theta_{n-1}|^p\\
&\le C\|S\|_{L^\infty}\left(\frac{2^n}{H}\right)^{p-1}\int_{t_{n-1}}^\infty\left(\int|\theta_{n-1}|^2\right)^{\frac{p(1-\gamma)}{2}}\left(\int|\nabla\theta_{n-1}|^2\right)^{\frac{p\gamma}{2}}\\
&\le C\|S\|_{L^\infty}\frac{2^{n(p-1)}}{H^{p-1}}c_{n-1}^\frac{p}{2}.
\end{align*}
Hence we conclude from \eqref{2.0.5-3} that
\begin{align}\label{2.0.5-4}
c_{n}\le\frac{C}{t_0}\cdot\frac{2^{n(p-1)+1}}{H^{p-2}}C_{n-1}^\frac{p}{2}+C\|S\|_{L^\infty}\frac{2^{n(p-1)}}{H^{p-1}}c_{n-1}^\frac{p}{2}.
\end{align}
We choose $H$ in \eqref{2.0.5-4} large enough so that
\begin{align}\label{value of H}
H=C\left(\frac{c_0^\frac{1}{2}}{t^\frac{d}{2}}+\|S\|_{L^\infty}^{-\frac{d}{d+4}}c_0^\frac{1}{2}\right),
\end{align}
then the nonlinear iteration inequality \eqref{2.0.5-4} implies that $c_n$ converges to 0 as $n\rightarrow\infty$. 

Hence $\theta(x,t_0)\le H$ for almost every $x\in\mathbb{T}^d$. Applying the same procedure to $-\theta$ gives a lower bound for $\theta$. To show that \eqref{2.0.5} holds, we need to estimate the term $c_0$. Using \eqref{energy inequality_2}, we have
\begin{align}\label{bound on c0}
c_0\le C\left(\|\theta_0\|_{L^2}^2+t_0\|S\|_{L^\infty}^2\right).
\end{align}
We combine \eqref{bound on c0} with \eqref{value of H} to obtain, for $t\le1$, 
\begin{align}\label{bound on theta for t<1}
\theta(t,x)\le C(\|\theta_0\|_{L^2}+\|S\|_{L^\infty})(\|S\|_{L^\infty}^{-\frac{d}{d+4}}+t^{-\frac{d}{2}}).
\end{align}
On the other hand, we fix $t=1$ in \eqref{bound on theta for t<1} and shift it by $t-1$ in time for $t>1$. Since the equation \eqref{2.0.1} is autonomous, we obtain, for $t>1$,
\begin{align}\label{bound on theta for t>1 1}
\theta(t,x)\le C(\|\theta(t-1,\cdot)\|_{L^2}+\|S\|_{L^\infty})(\|S\|_{L^\infty}^{-\frac{d}{d+4}}+1).
\end{align}
Using the energy inequality \eqref{energy inequality_1},
\begin{align*}
\|\theta(t-1,\cdot)\|_{L^2}^2\le C\left[e^{-(t-1)}\|\theta_0\|_{L^2}^2+\|S\|_{L^\infty}\right].
\end{align*}
Combining the above with \eqref{bound on theta for t>1 1} and using the fact that $e^{-(t-1)}\le\frac{1}{t}$ for $t>0$, we obtain, for $t>1$, 
\begin{align}\label{bound on theta for t>1 2}
\theta(t,x)\le C(t^{-\frac{1}{2}}\|\theta_0\|_{L^2}+\|S\|_{L^\infty})(\|S\|_{L^\infty}^{-\frac{d}{d+4}}+1).
\end{align}
Combining \eqref{bound on theta for t<1} and \eqref{bound on theta for t>1 2}, we conclude that \eqref{2.0.5} holds for $t>0$. 
\end{proof}

\noindent {\bf Step 2:} Next, we show that $\theta$ satisfies the {\it first energy inequality}, namely for any $0 <r<R$ and $h\in\R$, we have
\begin{align}\label{2.0.6}
&\|(\theta-h)_{+}\|^2_{L^\infty_t L^2_x(Q_r)}+\|\nabla(\theta-h)_{+}\|^2_{L^2_{t,x}(Q_r)}\notag\\
&\le\frac{CR}{(R-r)^2}\|(\theta-h)_{+}\|^{2-\frac{2}{d+2}}_{L^2_{t,x}(Q_R)}\|(\theta-h)_{+}\|^\frac{2}{d+2}_{L^\infty_{t,x}(Q_R)}+CR^y\|S\|^2_{L^\infty}|\{\theta>h\}\cap Q_R|^{1-\frac{1}{d+2}}\notag\\
&\qquad+\frac{CR^{\frac{y+1}{2}}\|S\|_{L^\infty}}{(R-r)}\|(\theta-h)_{+}\|_{L^2_{t,x}(Q_R)}|\{\theta>h\}\cap Q_R|^{\frac{1}{2}-\frac{1}{d+2}}
\end{align}
where $y=\frac{2(d+2)}{d}+1$, $C=C(d,\|V_{ij}\|_{L^\infty_t BMO_{x}})>0$ is a positive constant, and we have denoted $Q_{\rho} =[t_0-\rho,t_0]\times B_{\rho}(x_0)$ for $\rho>0$ and an arbitrary $(t_0,x_0)\in(0,\infty)\times\mathbb{T}^d$. Notice that by \eqref{2.0.5}, the right side of \eqref{2.0.6} is finite.
\begin{proof}
We follow the method for proving Lemma~2.6 in \cite{FV11a} and the only difference here comes from the extra forcing term $S$. Fix $h>0$ and let $0<r<R$ be such that $\frac{t_0}{2}-R^2>0$. Define $\eta(t,x)\in C^\infty_0((0,\infty)\times\mathbb{T}^d)$ to be a smooth cutoff function such that 
\begin{itemize}
\item[$\cdot$] $0\le\eta\le1$ in $(0,\infty)\times\mathbb{T}^d$;
\item[$\cdot$] $\eta\equiv1$ in $Q_r(x_0,t_0)$ and $\eta\equiv0$ in cl$\{Q^c_R(x_0,t_0)\cap\{(t,x):t\le t_0\}$;
\item[$\cdot$] $\dis|\nabla\eta|\le\frac{C}{R-r}$, \,\,$\dis|\nabla\nabla\eta|\le\frac{C}{(R-r)^2}$, \,\,$\dis|\partial_t\eta|\le\frac{C}{(R-r)^2}$ in $Q_R(x_0,t_0)/Q_r(x_0,t_0)$,
\end{itemize}
for some positive constant $C$. Define $t_1=t_0-R^2$ and let $t_2\in[t_0-r^2, t_0]$ be arbitrary. Multiply \eqref{2.0.1} by $(\theta-h)_{+}\eta^2$ and then integrate on $[t_1, t_2]\times\mathbb{T}^d$ to obtain
\begin{align}\label{2.0.7}
\int_{t_1}^{t_2}\!\!\!\int\partial_t((\theta-h)_{+}^2)\eta^2dxdt&-2\int_{t_1}^{t_2}\!\!\!\int\partial_{jj}(\theta-h)_{+}(\theta-h)_{+}\eta^2dxdt+\int_{t_1}^{t_2}\!\!\!\int\partial_i V_{ij}\partial_j ((\theta-h)_{+}^2)\eta^2dxdt\notag\\
&=\int_{t_1}^{t_2}\!\!\!\int(\theta-h)_{+}\eta^2Sdxdt.
\end{align}
From the estimates as shown in \cite{FV11a}, it follows from \eqref{2.0.7} that
\begin{align}\label{2.0.8}
&\|(\theta-h)_{+}\|^2_{L^\infty_t L^2_x(Q_r)}+\|\nabla(\theta-h)_{+}\|^2_{L^2_{t,x}(Q_r)}\notag\\
&\le\frac{CR}{(R-r)^2}\|(\theta-h)_{+}\|^{2-\frac{2}{d+2}}_{L^2_{t,x}(Q_R)}\|(\theta-h)_{+}\|^\frac{2}{d+2}_{L^\infty_{t,x}(Q_R)}+\left|\int_{t_1}^{t_2}\!\!\!\int(\theta-h)_{+}\eta^2Sdxdt\right|.
\end{align}
Using the Gagliardo-Nirenberg-Sobolev inequality \eqref{2a} for $\eta(\theta-h)_+\in H^1_0$, the second term on the right side of \eqref{2.0.8} can be bounded by
\begin{align*}
\left|\int_{t_1}^{t_2}\!\!\!\int(\theta-h)_{+}\eta^2Sdxdt\right|&\le\|S\|_{L^\infty}\left(\int_{t_1}^{t_2}\!\!\!\int|\eta(\theta-h)_{+}|^\frac{2d}{d-2}\right)^\frac{d-2}{2d}|\{\theta>h\}\cap Q_R|^{\frac{1}{2}+\frac{1}{d}}\notag\\
&\le C\|S\|_{L^\infty}\left(\int\!\!\!\int_{Q_R}|\nabla(\eta(\theta-h))_{+}|^2\right)^\frac{1}{2}|\{\theta>h\}\cap Q_R|^{\frac{1}{2}+\frac{1}{d}}\notag\\
&\le C\|S\|_{L^\infty}\left(\int\!\!\!\int_{Q_R}|\nabla(\theta-h)_{+}|^2\right)^\frac{1}{2}|\{\theta>h\}\cap Q_R|^{\frac{1}{2}-\frac{1}{2(d+2)}}|Q_R|^{\frac{1}{2(d+2)}+\frac{1}{d}}\notag\\
&\qquad+\frac{C\|S\|_{L^\infty}}{(R-r)}\left(\int\!\!\!\int_{Q_R}|(\theta-h)_{+}|^2\right)^\frac{1}{2}|\{\theta>h\}\cap Q_R|^{\frac{1}{2}-\frac{1}{d+2}}|Q_R|^{\frac{1}{d+2}+\frac{1}{d}}\notag\\
&\le C\|S\|_{L^\infty}\left(\int\!\!\!\int_{Q_R}|\nabla(\theta-h)_{+}|^2\right)^\frac{1}{2}|\{\theta>h\}\cap Q_R|^{\frac{1}{2}-\frac{1}{2(d+2)}}R^{\frac{d+2}{d}+\frac{1}{2}}\notag\\
&\qquad+\frac{C\|S\|_{L^\infty}}{(R-r)}\left(\int\!\!\!\int_{Q_R}|(\theta-h)_{+}|^2\right)^\frac{1}{2}|\{\theta>h\}\cap Q_R|^{\frac{1}{2}-\frac{1}{d+2}}R^{\frac{d+2}{d}+1},
\end{align*}
and hence using H\"{o}lder's inequality, the term $\dis\left(\int\!\!\!\int_{Q_R}|\nabla(\theta-h)_{+}|^2\right)^\frac{1}{2}$ can be absorbed by the left side of \eqref{2.0.8} and hence \eqref{2.0.6} follows.
\end{proof}

\noindent {\bf Step 3:} We give an estimate on the supremum of $\theta$ on a half cylinder in terms of the supremum on the full cylinder. Assume that $h_0\le\sup_{Q_{r_0}}\theta$, where $r_0>0$ is arbitrary, then we have
\begin{align}
\sup_{Q_{\frac{r_0}{2}}}\theta\le h_0+C\left(\frac{|\{\theta>h_0\}\cap Q_{r_0}|^\frac{1}{d+2}}{r_0}\right)^\frac{1}{2}\left(\sup_{Q_{r_0}}\theta-h_0\right)\label{2.0.9}
\end{align}
for some positive constant $C=C(d,\|V_{ij}\|_{L^\infty_t BMO_{x}})$.
\begin{proof}
To facilitate the proof, we first introduce the following notations:
\begin{itemize}
\item[$\cdot$] $A(h,r)=\{\theta>h\}\cap Q_r$,
\item[$\cdot$] $a(h,r)=|A(h,r)|$,
\item[$\cdot$] $b(h,r)=\|(\theta-h)_{+}\|^2_{L^2_{t,x}(Q_r)},$
\item[$\cdot$] $M(r)=\sup_{Q_r}\theta,$
\item[$\cdot$] $m(r)=\inf_{Q_r}\theta.$
\end{itemize}
Let $0<r<R$ and $0<h<H$. By the definitions of $a$ and $b$, we have
\begin{align}
a(H,r)\le\frac{b(h,r)}{(H-h)^2}.\label{2.0.10}
\end{align}
Following the proof of Lemma~2.10 in \cite{FV11a}, using \eqref{2.0.10} and the first energy inequality \eqref{2.0.6} as proved in Step~2, we have
\begin{align}\label{2.0.11}
b(h,r)\le &Ca(h,r)^\frac{2}{d+2}\frac{R}{(R-r)^2}b(h,R)^{1-\frac{1}{d+2}}\|(\theta-h)_{+}\|^{\frac{2}{d+2}}_{L^\infty_{t,x}(Q_R)}\notag\\
&+Ca(h,r)^\frac{2}{d+2}\|S\|^2_{L^\infty}\frac{R^y}{(H-h)^2}b(h,R)^{1-\frac{1}{d+2}}(H-h)^\frac{2}{d+2}\notag\\
&+Ca(h,r)^\frac{2}{d+2}\|S\|_{L^\infty}\frac{R^{\frac{y+1}{2}}}{(H-h)(R-r)}b(h,R)^{1-\frac{1}{d+2}}(H-h)^\frac{2}{d+2}.
\end{align}
By combining \eqref{2.0.10} and \eqref{2.0.11}, we obtain
\begin{align}\label{2.0.12}
b(H,r)\le&\frac{CR}{(H-h)^\frac{4}{d+2}(R-r)^2}b(h,R)^{1+\frac{1}{d+2}}\|(\theta-h)_{+}\|^{\frac{2}{d+2}}_{L^\infty_{t,x}(Q_R)}\notag\\
&+\frac{CR^y\|S\|^2_{L^\infty}}{(H-h)^\frac{4}{d+2}(H-h)^2}b(h,R)^{1+\frac{1}{d+2}}(H-h)^\frac{2}{d+2}\notag\\
&+\frac{CR^{\frac{y+1}{2}}\|S\|_{L^\infty}}{(H-h)^\frac{4}{d+2}(H-h)(R-r)}b(h,R)^{1+\frac{1}{d+2}}(H-h)^\frac{2}{d+2}.
\end{align}
We now apply the De Giorgi iteration method based on \eqref{2.0.12}. Let $r_n=\frac{r_0}{2} + \frac{r_0}{2^{n+1}}$, $h_n = h_\infty-\frac{(h_\infty-h_0)}{2^n}$, and $b_n = b(h_n, r_{n+1})$, for all $n\ge0$, where $r_0$ and $h_0$ are as given and $h_\infty>0$ is to be chosen later. By letting $H=h_{n+1}$, $h=h_n$, $r=r_{n+2}$, and $R=r_{n+1}$ in \eqref{2.0.12}, we have
\begin{align}\label{2.0.13}
b_{n+1}\le\left[\frac{C(M(r_0)-h_0)^\frac{2}{d+2}}{(h_\infty-h_0)^\frac{4}{d+2}r_0}+\frac{Cr^y_0\|S\|^2_{L^\infty}}{(h_\infty-h_0)^{\frac{2}{d+2}+2}}+\frac{Cr^{\frac{y}{2}-1}_0\|S\|_{L^\infty}}{(h_\infty-h_0)^{\frac{2}{d+2}+1}}\right]2^{n(2+\frac{4}{d+1})}b_n^{1+\frac{1}{d+2}}.
\end{align}
We let $B=2^{4+2(d+2)}$ and choose $h_\infty$ large enough so that
\begin{align*}
\left[\frac{C(M(r_0)-h_0)^\frac{2}{d+2}}{(h_\infty-h_0)^\frac{4}{d+2}r_0}+\frac{Cr^y_0\|S\|^2_{L^\infty}}{(h_\infty-h_0)^{\frac{2}{d+2}+2}}+\frac{Cr^{\frac{y}{2}-1}_0\|S\|_{L^\infty}}{(h_\infty-h_0)^{\frac{2}{d+2}+1}}\right]b_0^{1+\frac{1}{d+2}}\le\frac{1}{B},
\end{align*}
then by induction, we obtain from \eqref{2.0.13} that $\dis b_n\le\frac{b_0}{B^n}$ for all $n\in\mathbb{N}$. The rest follows from the argument given in \cite{FV11a} and we omit the details here.
\end{proof}

\noindent {\bf Step 4:} We have the following {\it second energy inequality} in controlling the
possible growth of level sets of the solution: fix an arbitrary $x_0\in\mathbb{T}^d$, let $h\in\R$, $0<r<R$, and $0<t_1<t_2$. Then we have
\begin{align}\label{2.0.14}
\|(\theta(t_2,\cdot)-h)_{+}\|^2_{L^2(B_r)}\le\|(\theta(t_1,\cdot)-h)_{+}\|^2_{L^2(B_r)}+\frac{C_0 R^d(t_2-t_1)}{(R-r)^2}\|(\theta-h)_+\|^2_{L^\infty((t_1,t_2)\times B_R)},
\end{align}
where $C_0=C_0(d,\|V_{ij}\|_{L^\infty_t BMO_{x}},\|S\|_{L^\infty})$ is a sufficiently large positive constant. Notice that by \eqref{2.0.5}, the right side of \eqref{2.0.14} is finite.
\begin{proof}
Similar to the first energy inequality, we follow the method for proving Lemma~2.11 in \cite{FV11a} and the only difference here comes from the forcing term $S$. Given $h,r,R,t_1,t_2$, we define $\eta\in C^\infty_0$ to be a smooth cutoff function such that 
\begin{itemize}
\item[$\cdot$] $0\le\eta\le1$;
\item[$\cdot$] $\eta\equiv1$ on $B_r$ and $\eta\equiv0$ on $B_R^c$;
\item[$\cdot$] $\dis|\nabla\eta(x)|\le\min\Big\{\frac{C}{R-r},\frac{CR}{(R-r)^2}\|(\theta-h)_{+}\|_{L^\infty((t_1,t_2)\times B_R)}\Big\}$ for all $x\in\mathbb{T}^d$,
\end{itemize}
for some positive constant $C$. Multiply \eqref{2.0.1} by $(\theta-h)_{+}\eta^2$ and integrate on $[t_1, t_2]\times\mathbb{T}^d$, it follows from the estimates given in \cite{FV11a} that
\begin{align}\label{2.0.15}
\|(\theta(t_2,\cdot)-h)_{+}\|^2_{L^2(B_r)}\le&\|(\theta(t_1,\cdot)-h)_{+}\|^2_{L^2(B_r)}+\frac{C R^d(t_2-t_1)}{(R-r)^2}\|(\theta-h)_+\|^2_{L^\infty((t_1,t_2)\times B_R)}\notag\\
&+\left|\int_{t_1}^{t_2}\!\!\!\int(\theta-h)_{+}\eta^2Sdxdt\right|,
\end{align}
where $C=C(d,\|V_{ij}\|_{L^\infty_t BMO_{x}})$ is a positive constant. Using the Gagliardo-Nirenberg-Sobolev inequality \eqref{2a} for $\eta\in C^\infty_0$, we bound the far right side of the above as follows.
\begin{align}\label{2.0.16}
\left|\int_{t_1}^{t_2}\!\!\!\int(\theta-h)_{+}\eta^2Sdxdt\right|&\le\|(\theta-h)_+\|_{L^\infty((t_1,t_2)\times B_R)}\|S\|_{L^\infty}(t_1-t_2) \left(\int_{B_R}|\eta|^\frac{2d}{d-2}\right)^\frac{d-2}{2d}|B_R|^{\frac{1}{2}-\frac{1}{d}}\notag\\
&\le CR^{\frac{d}{2}-1}\|(\theta-h)_+\|_{L^\infty((t_1,t_2)\times B_R)}\|S\|_{L^\infty}(t_1-t_2)\left(\int_{B_R}|\nabla\eta|^\frac{1}{2}\right)^\frac{1}{2}\notag\\
&\le CR^{\frac{d}{2}-1+\frac{d}{2}}\|(\theta-h)_+\|^2_{L^\infty((t_1,t_2)\times B_R)}\|S\|_{L^\infty}(t_1-t_2)\frac{CR}{(R-r)^2}\notag\\
&=\frac{C\|S\|_{L^\infty}R^d(t_2-t_1)}{(R-r)^2}\|(\theta-h)_+\|^2_{L^\infty((t_1,t_2)\times B_R)}.
\end{align}
By using \eqref{2.0.16} on \eqref{2.0.15}, we can choose $C_0=C_0(d,\|V_{ij}\|_{L^\infty_t BMO_{x}},\|S\|_\infty)>0$ sufficiently large enough so that \eqref{2.0.14} holds.
\end{proof}

\noindent {\bf Step 5:} Using the second energy inequality \eqref{2.0.14}, we can bound $|\{\theta(t_2,\cdot)\ge H\}\cap B_R|/|B_R|$ whenever $|\{\theta(t_1,\cdot)\ge H\}\cap B_r|/|B_r|\le\frac{1}{2}$. Fix $\kappa_0=(\frac{4}{5})^\frac{1}{d}$, let $n_0\ge2$ be the least integer such that $\frac{2^{n_0}}{2^{n_0}-2}\le\sqrt{\frac{6}{5}}$, and let $\delta_0=\frac{(1-\kappa_0)^2}{12C_0\kappa_0}$ where $C_0$ is the constant from \eqref{2.0.14}. For $t_1,R>0$, if 
\begin{align*}
|\{\theta(t_1,\cdot)\ge H\}\cap B_r|\le\frac{1}{2}|B_r|,
\end{align*}
then for all $t_2\in[t_1,t_1+\delta_0r^2]$ we have
\begin{align}\label{2.0.17}
|\{\theta(t_2,\cdot)\ge H\}\cap B_R|\le\frac{7}{8}|B_R|,
\end{align}
where we define $r=r_0R$, \,\,\,$M=\sup_{(t_1,t_1+\delta_0R^2)\times B_R}\theta$, \,\,\,$m=\inf_{(t_1,t_1+\delta_0R^2)\times B_R}\theta$, \,\,\,$h=\frac{(M+m)}{2}$ and \newline$H=M-\frac{(M-m)}{2^{n_0}}$.
\begin{proof}
By using \eqref{2.0.14}, the proof of \eqref{2.0.17} follows by the same argument given by the proof of Lemma~2.12 in \cite{FV11a} and we omit the details here.
\end{proof}

\noindent {\bf Step 6:} Applying Step~1 to Step~5, the proof of Proposition~3.2 now follows by showing that there exists $\beta\in(0,1)$ independent of $R$, such that 
\begin{align}\label{2.0.18}
osc(Q_1)\le\beta osc(Q_2), 
\end{align} 
where $Q_1 = [t_1,t_1 +\delta_0r^2]\times B_r$, $Q_2 = [t_1,t_1 +\delta_0R^2]\times B_R$ and $osc(Q)=\sup_{Q}\theta-\inf_{Q}\theta$. Here $\kappa_0,\delta_0,n_0,M,m,h,H,r,R$ are defined as in Step 5, and we recall that $t_1 >0$ and $R>0$ are arbitrary. We refer the reader to (\cite{FV11a}, pp. 293--294) for details in proving \eqref{2.0.18}. The estimate \eqref{2.0.18} implies the H\"{o}lder regularity of the solution (the H\"{o}lder exponent $\alpha\in(0, 1)$ may be calculated explicitly from $\beta$) which finishes the proof of Proposition~3.2.
\end{proof}

\noindent{\bf The Nonlinear Problem}

\medskip

\noindent We now focus back on the non-linear problem \eqref{0.1} and give the proof of Theorem~\ref{existence of smooth solutions sub Thm}. 

\begin{proof}[Proof of Theorem~\ref{existence of smooth solutions sub Thm}]
To begin with, we notice that given $\theta_0\in L^2$ and $\nu\ge0$, there exists a global-in-time Leray-Hopf weak solution $\theta^\nu$ of \eqref{0.1} evolving from $\theta_0$ (a proof for it can be found in \cite{FV11a}). Using the same method as in proving \eqref{2.0.5}, we have $\theta^\nu\in L^\infty_{t ,x}$ and it follows from the Calderón-Zygmund theory of singular integrals that $T^\nu_{ij}\theta=: V^\nu_{ij}\in L^\infty ([t_0 ,\infty); BMO)$, for any $t_0 > 0$ and $\nu\ge0$, where $i,j \in\{1,...,d\}$. Therefore, we may treat \eqref{0.1} as a linear evolution equation (see also \cite{CV10}, \cite{CW09}, \cite{FV11a}), where the divergence-free velocity field $u$ is given, and $u\in L^2((0,\infty); L^2)\cap L^\infty([t_0, \infty); BMO^{-1})$, for any $t_0 > 0$. This is precisely the setting of Proposition~3.2 for the linear evolution equation and it can be applied to the nonlinear problem \eqref{0.1} to give H\"{o}lder regularity of the solution. Finally, since H\"{o}lder regularity is sub-critical for the natural scaling of \eqref{0.1}, one may bootstrap to prove that the solution is in a higher regularity class. We refer to \cite{FV11a} for further details and conclude the proof of Theorem~\ref{existence of smooth solutions sub Thm}.
\end{proof}

\section{Uniform bounds on smooth solutions and proof of Theorem~1.2}\label{Uniform bounds on smooth solutions and proof of Theorem 1.2}

\subsection{Uniform bounds on smooth solutions}

\noindent We have the following uniform $H^s$-bound on smooth solutions to \eqref{0.1} which will be used in proving Theorem~\ref{Convergence of solutions as nu goes to 0}.

\begin{theorem}\label{uniform bound on solution Thm}
Assume that the hypotheses and notations of Theorem~\ref{existence of smooth solutions sub Thm} are in force. Then given $0<t_1<t_2$ and $s\ge0$, there exists a positive constant $C(C_{*},t_1,t_2,s,d,\kappa,S,\|\theta_0\|_{L^2})>0$ independent of $\nu$ such that
\begin{align}\label{2.1}
\sup_{t\in[t_1,t_2]}\|\theta^{\nu}(t,\cdot)\|_{ H^s}+\int_{t_1}^{t_2}\|\theta^{\nu}(t,\cdot)\|^2_{H^{s+1}}dt\le C(C_{*},t_1,t_2,s,d,\kappa,S,\|\theta_0\|_{L^2}),
\end{align}
where $C_{*}>0$ is the constant as stated in condition (1.13).
\end{theorem}

\begin{proof}[Proof of Theorem~\ref{uniform bound on solution Thm}] 
Fix $I=[t_1,t_2]$ for some $0<t_1<t_2$. By Theorem~\ref{existence of smooth solutions sub Thm}, there exists $\alpha\in(0,1)$ such that for each $\nu\ge0$, 
\begin{align}\label{2.2}
\theta^{\nu}\in L^{\infty}(I;L^2(\mathbb{T}^d)\cap L^2(I; H^1(\mathbb{T}^d))\cap L^{\infty}(I; C^{\alpha}(\mathbb{T}^d)).
\end{align}
Depending on the value of $\alpha$, we consider the following 2 cases:

\medskip

\noindent{\bf Case~1: $\alpha\in(0,\frac{1}{2}]$.} The proof is based on the one given in \cite{FV12} and we just need to take extra care of the forcing term $S$. It is given in the following steps:

\medskip

\noindent{\bf Step~(i):} Assume further that
\begin{align}
\theta^\nu\in L^2(I;  B^1_{p,2}(\mathbb{T}^d))
\end{align}
for some $p\ge2$, then we have
\begin{align}\label{2.2.00}
\theta^\nu\in\tilde{L}^2(I; B^1_{q,r}(\mathbb{T}^d))
\end{align}
for all $1\le r\le\infty$, and for all $q\in(p,m_\alpha p)$, where $m_\alpha=\frac{1-\alpha}{1-2\alpha}>1$.
\begin{proof}
Let $\Delta_j$ be as defined in Section~\ref{Preliminaries and notations}. We apply $\Delta_j$ to \eqref{0.1}, multiply by $\Delta_j\theta|\Delta_j\theta|^{q-2}$, integrate over $\mathbb{T}^d$, and use Proposition~29.1 in \cite{L02} (also refer to \cite{CMZ07}) to obtain, for $j\in\mathbb{Z}$,
\begin{align}\label{2.2.0}
\frac{1}{q}\frac{d}{dt}\|\Delta_j\theta^\nu\|^q_{L^q}+C2^{2j}\|\Delta_j\theta^\nu\|^q_{L^q}\le\left|\int\Delta_j(u^\nu\cdot\theta^\nu)\Delta_j\theta^\nu|\Delta_j\theta^\nu|^{q-2}\right|+\left|\int\Delta_j(S)\Delta_j\theta^\nu|\Delta_j\theta^\nu|^{q-2}\right|,
\end{align}
where $C=C(d,q)>0$ is a positive constant independent of $\nu$. Using H\"{o}lder inequality, the second term on the right side of \eqref{2.2.0} is bounded by $\dis\|\Delta_j\theta^\nu\|^{q-1}_{L^q}\|\Delta_j(S)\|_{L^q}$. Hence applying the similar estimates given in the proof of Lemma~2 in \cite{FV12} pp. 259--261, we obtain 
\begin{align}\label{2.2.01}
\frac{d}{dt}\|\Delta_j\theta^\nu\|_{L^q}+C2^{2j}\|\Delta_j\theta^\nu\|_{L^q}\le &C\|\theta^\nu\|^{2-\frac{p}{q}}_{C^\alpha}2^{j(1-\alpha)}\sum_{k\le j} 2^{k(1-\frac{p}{q}-\alpha(1-\frac{p}{q}))}(2^k\|\Delta_k\theta^\nu\|_{L^p})^\frac{p}{q}\notag\\
&+C\|\theta^\nu\|^{2-\frac{p}{q}}_{C^\alpha}2^{j(2-\alpha-\frac{p}{q}-\alpha(1-\frac{p}{q}))}\sum_{|j-k|\le2}(2^k\|\Delta_k\theta^\nu\|_{L^p})^\frac{p}{q}\notag\\
&+C\|\theta^\nu\|^{2-\frac{p}{q}}_{C^\alpha}2^j\sum_{k\ge j-1}2^{k(1-\alpha-\frac{p}{q}-\alpha(1-\frac{p}{q}))}(2^k\|\Delta_k\theta^\nu\|_{L^p})^\frac{p}{q}\notag\\
&+C\|\Delta_j S\|_{L^q},
\end{align}
where $C=C(C_*)>0$ is a positive constant independent of $\nu$ and $C_*$ is defined in condition (1.13). Applying Gr\"{o}nwall's inequality on \eqref{2.2.01},
\begin{align}\label{2.2.02}
\|\Delta_j\theta^\nu(t)\|_{L^q}\le &e^{-c2^{2j(t-t_1)}}\|\Delta_j\theta^\nu(t_1)\|_{L^q}\notag\\
&+C\|\theta^\nu\|^{2-\frac{p}{q}}_{L^\infty(I;C^\alpha)}2^{j(1-\alpha)}\sum_{k\le j}2^{k(1-\frac{p}{q}-\alpha(1-\frac{p}{q}))}\Theta_{j,k}(t)\notag\\
&+C\|\theta^\nu\|^{2-\frac{p}{q}}_{L^\infty(I;C^\alpha)}2^{j(2-\alpha-\frac{p}{q}-\alpha(1-\frac{p}{q}))}\sum_{|k-j|\le2}\Theta_{j,k}(t)\notag\\
&+C\|\theta^\nu\|^{2-\frac{p}{q}}_{L^\infty(I;C^\alpha)}2^j\sum_{k\ge j-1}2^{k(1-\alpha-\frac{p}{q}-\alpha(1-\frac{p}{q}))}\Theta_{j,k}(t)\notag\\
&+\|\Delta_j S\|_{L^q}\int_{t_1}^te^{-c(t-\tau)2^{2j}}d\tau,
\end{align}
where 
$$\Theta_{j,k}(t)=\int_{t_0}^t e^{-c(t-\tau)2^{2j}}(2^k\|\Delta_k\theta^\nu(s)\|_{L^p}^\frac{p}{q})d\tau.$$
We take the $L^2(I)$ norm of \eqref{2.2.02} and apply the similar estimates given in \cite{FV12} pp. 260--261 to obtain
\begin{align}\label{2.2.03-}
\|\Delta_j\theta^\nu(t)\|_{L^2(I;L^q)}\le &C\|\theta^\nu(t_1)\|_{L^q}^\frac{p}{q}\|\theta^\nu(t_1)\|_{C^\alpha}^{1-\frac{p}{q}}(2^{-j\alpha(1-\frac{p}{q})}\min\{2^{-j},|I|^\frac{1}{2}\})\notag\\
&+C\|\theta^\nu\|_{L^\infty(I;C^\alpha)}\|\theta^\nu\|_{L^2(I; B^1_{p,2})}^\frac{p}{q}|I|^\frac{q-p}{2q}(2^{j(2-\alpha-\frac{p}{q}-\alpha(1-\frac{p}{q})})\min\{C2^{-2j},|I|\})\notag\\
&+C\|\Delta_j S\|_{L^q}|I|^\frac{1}{2}\min\{C2^{-j},|I|\}.
\end{align}
Multiply the above on both sides by $2^j$ and take an $\ell^r(\mathbb{Z})$-norm,
\begin{align}\label{2.2.03}
\|\theta^\nu\|_{\tilde{L}^2(I; B^1_{q,r})}\le &C\|\theta^\nu(t_1)\|_{L^q}^\frac{p}{q}\|\theta^\nu(t_1)\|_{C^\alpha}^{1-\frac{p}{q}}\|(2^{j(1-\alpha(1-\frac{p}{q}))}\min\{2^{-j},|I|^\frac{1}{2}\})\|_{l^r(\mathbb{Z})}\notag\\
&+C\|\theta^\nu\|_{L^\infty(I;C^\alpha)}\|\theta^\nu\|_{L^2(I; B^1_{p,2})}^\frac{p}{q}|I|^\frac{q-p}{2q}\|(2^{j(3-\alpha-\frac{p}{q}-\alpha(1-\frac{p}{q})})\min\{C2^{-2j},|I|\})\|_{l^r(\mathbb{Z})}\notag\\
&+C\|2^j\|\Delta_j S\|_{L^q}|I|^\frac{1}{2}\min\{C2^{-j},|I|\}\|_{\ell^r(\mathbb{Z})}.
\end{align} 
Since $q\in(p,m_\alpha p)$, the two $\ell^r$ norms on the right side of the above estimate are finite for any $1\le r\le\infty$. On the other hand, the last term on the right side of \eqref{2.2.03} can be bounded by $\|S\|_{L^\infty}|I|^\frac{1}{2}$. Hence we have $\theta^\nu\in\tilde{L}^2(I; B^1_{q,r})$ which finishes the proof of \eqref{2.2.00}.
\end{proof}
\noindent{\bf Step~(ii):} Assume that $\theta^\nu$ satisfies \eqref{2.2}, we then have
\begin{align}\label{2.2.04}
\nabla\theta^{\nu}\in L^2(I; L^{\infty}(\mathbb{T}^d)).
\end{align}
\begin{proof}
We follow the proof of Lemma~3 in \cite{FV12} pp. 262--263. First, we note that $ H^1= B^1_{2,2}$, so we may apply Step~(i) with $p=2$ and obtain that $\theta\in L^2(I; B^1_{q,2})$ for any $q\in(2,2m_\alpha)$. Since $m_\alpha> 1$, we may bootstrap and apply Step~(i) once more to obtain that $\theta\in L^2(I; B^1_{q,2})$ for all $q\in(2,2m^2_{\alpha})$. For any fixed $p>2$, we have $m_{\alpha}>1$ and hence $m_{\alpha}^k\rightarrow\infty$ as $k\rightarrow\infty$. By iterating Step~(i) finitely many times,  we obtain 
\begin{align}\label{2.2.05}
\mbox{$\theta^\nu\in\tilde{L}^2(I; B^1_{p,r})$ for all $r\in[1,\infty]$.}
\end{align}
Fix $p$ large enough (which will be explicitly chosen later), and let $q = \frac{p(1 + m_\alpha)}{2}$. From the estimate \eqref{2.2.03-}, for any $\varepsilon>0$,
\begin{align}\label{2.2.06}
2^{j(1+\varepsilon)}\|\Delta_j\theta^\nu(t)\|_{L^2(I;L^q)}\le &C\|\theta^\nu(t_1)\|_{L^q}^\frac{p}{q}\|\theta^\nu(t_1)\|_{C^\alpha}^{1-\frac{p}{q}}\min\{C2^{j(\varepsilon-\alpha(1-\frac{p}{q}))},|I|^\frac{1}{2}2^{j(1+\varepsilon-\alpha(1-\frac{p}{q}))}\}\notag\\
&+C\|\theta^\nu\|_{L^\infty(I;C^\alpha)}\|\theta^\nu\|_{L^2(I; B^1_{p,2})}^\frac{p}{q}|I|^\frac{q-p}{2q}\notag\\
&\qquad\times\min\{C2^{j(\varepsilon+1-\frac{p}{q}-\alpha(2-\frac{p}{q})},2^{j(\varepsilon+3-\frac{p}{q}-\alpha(2-\frac{p}{q})}|I|\})\notag\\
&+C\|\Delta_j S\|_{L^q}|I|^\frac{1}{2}\min\{C2^{j(\varepsilon-1)},2^{j(1+\varepsilon)}|I|\},
\end{align}
Choose
\begin{align*}
\varepsilon=\frac{1}{2}\min\left\{\frac{\alpha^2}{2-3\alpha},\frac{(1-2\alpha)(2-3\alpha-\alpha^2)}{(1-\alpha)(2-3\alpha)},1\right\},
\end{align*}
then $\varepsilon>0$ for all $\alpha\in(0,\frac{1}{2})$. By taking the $\ell^r$ norm of \eqref{2.2.06} and using Besov embedding theorem \eqref{2d}, we have
\begin{align}\label{2.2.07}
\theta^\nu\in\tilde{L}^2(I; B^{1+\varepsilon}_{q,1})\subset L^2(I; B^{1+\varepsilon}_{q,1})\subset L^2(I; B^{1+\varepsilon-\frac{2d}{p+pm_\alpha}}_{\infty,1}).
\end{align}
We pick $p>2$ so that $\varepsilon-\frac{2d}{p+pm_\alpha}=0$, then we obtain from \eqref{2.2.07} that
\begin{align*}
\nabla\theta^\nu\in L^2(I; B^0_{\infty,1}).
\end{align*}
Lastly, from condition \eqref{2.2}, we have $\nabla\theta^\nu\in L^2(I;L^2\cap B^0_{\infty,1})$, hence \eqref{2.2.04} follows from the the borderline Sobolev embedding theorem that $L^2\cap B^0_{\infty,1}\subset B^0_{\infty,1}\subset L^\infty$.
\end{proof}

\noindent Using the results obtained from Step~(i) and Step~(ii) as described above, we are ready to prove the bound \eqref{2.2} for $\alpha\in(0,\frac{1}{2}]$. First, by using \eqref{2.2.04}, we have
\begin{align}
\int_{t_1}^{t_2}\|\nabla\theta^{\nu}(t,\cdot)\|^2_{L^{\infty}}dt\le C(C_{*},t_1,t_2,d,\kappa,S,\|\theta_0\|_{L^2}),\label{2.3}
\end{align}
where $C(C_{*},t_1,t_2,d,\kappa,S,\|\theta_0\|_{L^2})>0$ is a positive constant independent of $\nu$. Next, by the condition (1.2), $u^{\nu}$ is divergence free and we have the a priori estimate derived from \eqref{0.1} that
\begin{align}
&\frac{1}{2}\frac{d}{dt}\|\nabla\theta^{\nu}(t,\cdot)\|_{L^2}+\kappa\|\Delta\theta^{\nu}(t,\cdot)\|_{L^2}\notag\\
&\le\left|\int\partial_{k}u^{\nu}_{j}\partial_{k}\theta^{\nu}\partial_{j}\theta^{\nu}\right|+\left|\int S\Delta\theta^{\nu}\right|\notag\\
&\le C_{*}\|\Delta\theta^{\nu}(t,\cdot)\|_{L^2}\|\nabla\theta^{\nu}(t,\cdot)\|_{L^2}\|\nabla\theta^{\nu}(t,\cdot)\|_{L^\infty}+\|\nabla\theta^{\nu}(t,\cdot)\|_{L^2}\|\nabla S(t,\cdot)\|_{L^2}\notag\\
&\le\frac{\kappa}{2}\|\Delta\theta^{\nu}(t,\cdot)\|^2_{L^2}+C(C_{*},\kappa,S)\|\nabla\theta^{\nu}(t,\cdot)\|^2_{L^2}\|\nabla\theta^{\nu}(t,\cdot)\|^2_{L^\infty}.\label{2.4}
\end{align}
By absorbing the term $\dis\frac{\kappa}{2}\|\Delta\theta^{\nu}(t,\cdot)\|^2_{L^2}$ on the left side of \eqref{2.4} and using the bound \eqref{2.3}, we obtain, for all $t\ge t_1$,
\begin{align}
\|\theta^{\nu}(t,\cdot)\|^2_{ H^1}\le \|\theta^{\nu}(t_2,\cdot)\|^2_{ H^1}e^{\int_{t_1}^{t_2}C(C_{*},\kappa,S)\|\nabla\theta^{\nu}(t,\cdot)\|_{L^{\infty}}^2dt}\le \|\theta^{\nu}(t_2,\cdot)\|^2_{ H^1}e^{C(C_{*},t_1,t_2,d,\kappa,S,\|\theta_0\|_{L^2})}.\label{2.5}
\end{align}
Since $\theta^\nu\in L^2([t_1,t_2]; H^1)$ and $L^2$ functions are finite a.e., $\|\theta^{\nu}(t_2,\cdot)\|_{ H^1}$ is finite for a.e. $t_2>0$, with bounds in terms of $\|\theta_0\|_{L^2}$ but independent of $\nu$. Hence \eqref{2.5} implies $\theta^{\nu}\in L^{\infty}([t_1,t_2];  H^1)\cap L^2([t_1,t_2];  H^2)$ with
\begin{align*}
\sup_{t\in[t_1,t_2]}\|\theta^{\nu}(t,\cdot)\|_{ H^1}+\int_{t_1}^{t_2}\|\theta^{\nu}(t,\cdot)\|^2_{  H^2(\R^d))}dt\le C(C_{*},t_1,t_2,d,\kappa,S,\|\theta_0\|_{L^2}).
\end{align*}
By further taking derivatives of the equation \eqref{0.1} and repeating the above argument, \eqref{2.1} also holds for all $s>1$ and we finish the proof of \eqref{2.1} for $\alpha\in(0,\frac{1}{2}]$.

\medskip

\noindent{\bf Case 2: $\alpha\in(\frac{1}{2},1)$.} We adopt the method given in \cite{FV12} pp. 257--258. Similar to Case~1, the goal is to prove that \eqref{2.2.04} holds for $\theta^\nu$. Once \eqref{2.2.04} is proved, same argument given in the previous case can then be applied which gives the bound \eqref{2.1}. First, note that if $\theta^\nu$ satisfies \eqref{2.2}, then $\theta^\nu\in L^\infty([t_0,\infty); B^{\alpha_p}_{p,\infty})$, where $\alpha_p=(1-\frac{2}{p})\alpha$ and $p\in[2,\infty)$ is fixed and to be chosen later. Then, for $j\in\mathbb{Z}$ fixed, we have
\begin{align}\label{2.5.01}
\frac{1}{p}\frac{d}{dt}\|\Delta_j\theta^\nu\|^p_{L^p}+C2^{2j}\|\Delta_j\theta^\nu\|^p_{L^p}\le\left|\int|\Delta_j\theta^\nu|^{p-2}\Delta_j\theta^\nu\Delta_j(u\cdot\nabla\theta^\nu)\right|+\left|\int\Delta_j(S)\Delta_j\theta^\nu|\Delta_j\theta^\nu|^{p-2}\right|.
\end{align}
Using the Bony paraproduct decomposition and the method given in \cite{FV12}, the first term on the right side of \eqref{2.5.01} can be bounded by $C2^{(2-2\alpha_p)j}\|\theta^\nu\|_{C^{\alpha_p}}\|\theta^\nu\|_{ B^{\alpha_p}_{p,\infty}}$, where $C>0$ is a constant which may depend on $C_*$ as in condition (1.13) but independent of $\nu$. On the other hand, the term $\dis\left|\int\Delta_j(S)\Delta_j\theta^\nu|\Delta_j\theta^\nu|^{p-2}\right|$ can be bounded by $\|\Delta_j\theta^\nu\|^{p-1}_{L^p}\|\Delta_j S\|_{L^p}$. By applying the bounds on \eqref{2.5.01}, using Gr\"{o}wall inequality and the Besov embedding theorem \eqref{2d}, we obtain
\begin{align*}
\theta^\nu\in L^\infty([t,\infty); B^{2\alpha_p}_{p,\infty})\subset L^\infty([t,\infty); B^{2\alpha-\frac{4\alpha+d}{p}}_{\infty,\infty})
\end{align*}
for all $t\ge t_1$. Choose $p>\frac{4+d}{2\alpha-1}$, then $2\alpha-\frac{4\alpha+d}{p}>1$. Since $L^\infty\cap B^{2\alpha-\frac{4\alpha+d}{p}}_{\infty,\infty}=C^{2\alpha-\frac{4\alpha+d}{p}}$, we conclude that \eqref{2.2.04} holds for $\theta^\nu$ and we finish the proof of \eqref{2.1} for $\alpha\in(\frac{1}{2},1)$.
\end{proof}

\subsection{Proof of Theorem~1.2}

The proof can be divided into two cases:

\medskip

\noindent{\bf Case 1: $s=0$.} Let $\phi=\theta^{\nu}-\theta$, then $\phi$ satisfies the following equation:
\begin{align}
\partial_t\phi+u^{\nu}\cdot\nabla\phi+(u^{\nu}-u)\cdot\nabla\theta=\kappa\Delta\phi.\label{5.1}
\end{align}
Multiply \eqref{5.1} by $\phi$ and integrate,
\begin{align}
\frac{1}{2}\frac{d}{dt}\|\phi(t,\cdot)\|_{L^2}^2+\frac{\kappa}{2}\|\nabla\phi(t,\cdot)\|_{L^2}^2=-\int(u^{\nu}-u)\cdot\nabla\theta\cdot\phi(t,x)dx.\label{5.2}
\end{align}
We estimate the right side of \eqref{5.2} as follows. For each $t>0$,
\begin{align}
\left|-\int(u^{\nu}-u)\cdot\nabla\theta\cdot\phi(t,x)dx\right|&\le\|(u^{\nu}-u)(t,\cdot)\|_{L^2}\|\phi(t,\cdot)\|_{L^2}\|\nabla\theta(t,\cdot)\|_{L^\infty}\notag\\
&\le\frac{\kappa}{4C_{*}^2}\|(u^{\nu}-u)(t,\cdot)\|^2_{L^2}+\frac{4C_{*}^2}{\kappa}\|\phi(t,\cdot)\|_{L^2}^2\|\nabla\theta(t,\cdot)\|_{L^\infty}^2,\label{5.3}
\end{align}
where $C_{*}>0$ is the constant as stated in condition (1.13). We focus on the term $\|(u^{\nu}-u)(t,\cdot)\|^2_{L^2}$ as in \eqref{5.3}, and for simplicity we sometime drop the variable $t$. Using Plancherel Theorem, for each $j$, 
\begin{align}
\|(u_j^{\nu}-u_j)(t,\cdot)\|^2_{L^2}&=\sum_{k\in\mathbb{Z}^d}|\widehat{(u_j^{\nu}-u_j)}(k)|^2\notag\\
&=\sum_{k\in\mathbb{Z}^d}|(\widehat{\partial_{i} T_{ij}^{\nu}}\widehat{\theta^{\nu}}-\widehat{\partial_{i} T^0_{ij}}\widehat{\theta})(k)|^2\notag\\
&\le\sum_{k\in\mathbb{Z}^d}|\widehat{\partial_{i} T_{ij}^{\nu}}|^2|\widehat\phi|^2(k)+\sum_{k\in\mathbb{Z}^3}|\widehat{\partial_{i} T_{ij}^{\nu}}-\widehat{\partial_{i} T^0_{ij}}|^2|\widehat\theta|^2(k)\notag\\
&\le\sum_{k\in\mathbb{Z}^d}|\widehat{T_{ij}^{\nu}}(k)|^2|\widehat{\nabla\phi}(k)|^2+\sum_{k\in\mathbb{Z}^3}|\widehat{T^{\nu}_{ij}}(k)-\widehat{T^0_{ij}}(k)|^2|\widehat{\nabla\theta}(k)|^2\notag\\
&:=I_1+I_2.\label{5.4}
\end{align}
Using the condition (1.13), the term $I_1$ can be estimated by
\begin{align}
I_1\le C_{*}^2\sum_{k\in\mathbb{Z}^d}|\widehat{\nabla\phi}(k)|^2=C_{*}^2\|\nabla\phi(t,\cdot)\|_{L^2}^2.\label{5.5}
\end{align}
For the term $I_2$, by Theorem~\ref{existence of smooth solutions sub Thm}, $\theta$ is smooth and $\|\nabla\theta(t,\cdot)\|_{L^2}<\infty$. So the condition (1.4) can be applied and we have
\begin{align}
\lim_{\nu\rightarrow0}I_2=\lim_{\nu\rightarrow0}\sum_{k\in\mathbb{Z}^d:k\neq0}|\widehat{T^{\nu}_{ij}}(k)-\widehat{T^0_{ij}}(k)|^2|\widehat{\nabla\theta}(k)|^2=0.\label{5.6}
\end{align}
We apply \eqref{5.5} on \eqref{5.4} to obtain
\begin{align}
\|(u^{\nu}-u)(t,\cdot)\|^2_{L^2}\le C_{*}^2\|\nabla\phi(t,\cdot)\|_{L^2}^2+I_2,\label{5.7}
\end{align}
and hence using \eqref{5.7} on \eqref{5.3},
\begin{align}
\left|-\int(u^{\nu}-u)\cdot\nabla\theta\cdot\phi(t,x)dx\right|\le\frac{\kappa}{4}\|\nabla\phi(t,\cdot)\|_{L^2}^2+\frac{\kappa}{4C_{*}^2}I_2+\frac{4C_{*}^2}{\kappa}\|\phi(t,\cdot)\|_{L^2}^2\|\nabla\theta(t,\cdot)\|_{L^\infty}^2.\label{5.8}
\end{align}
Applying \eqref{5.8} on \eqref{5.2}, using Gr\"{o}nwall's inequality, taking $\nu\rightarrow0$ and using \eqref{5.6}, for all $t>0$, we conclude that
\begin{align}
\lim_{\nu\rightarrow0}\|(\theta^{\nu}-\theta)(t,\cdot)\|_{L^2}^2=\lim_{\nu\rightarrow0}\|\phi(t,\cdot)\|_{L^2}^2=0.\label{5.9}
\end{align}

\noindent{\bf Case 2: $s>0$.} We apply the Gagliardo-Nirenberg interpolation inequality for homogeneous Sobolev space \eqref{2c} to obtain, for $s>0$ and $t>0$,
\begin{align}
\|(\theta^{\nu}-\theta)(t,\cdot)\|_{ H^s}\le C\|(\theta^{\nu}-\theta)(t,\cdot)\|_{L^2}^{\gamma}\|(\theta^{\nu}-\theta)(t,\cdot)\|_{ H^{s+1}}^{1-\gamma},
\end{align}
where $\gamma\in(0,1)$ depends on $s$, and $C>0$ is a positive constant which depends on $d$ but is independent of $\nu$. Using the bounds \eqref{2.1} as in Theorem~\ref{uniform bound on solution Thm}, the term $\|(\theta^{\nu}-\theta)(t,\cdot)\|_{ H^{s+1}}^{1-\gamma}$ is bounded uniformly in $\nu$ for all $t\ge\tau$. Hence by taking $\nu\rightarrow0$ and applying the $L^2$ convergence \eqref{5.9}, we have 
\begin{align*}
\mbox{$\dis\lim_{\nu\rightarrow0} \|(\theta^{\nu}-\theta)(t,\cdot)\|_{ H^s}=0$ for $t\ge\tau$,}
\end{align*}
which finishes the proof of Theorem~\ref{Convergence of solutions as nu goes to 0}.

\section{The MG Equations}\label{The MG Equations}

\subsection{The explicit symbol of the MG operator}

We now return to the magnetogeostrophic active scalar equation discussed in the introduction. Specifically, we are interested in the following active scalar equation in the domain $(0,\infty)\times\mathbb{T}^3=(0,\infty)\times[0,2\pi]^3$ (with periodic boundary conditions):
\begin{align}
\label{1.1} \left\{ \begin{array}{l}
\partial_t\theta^{\nu}+u^{\nu}\cdot\nabla\theta^{\nu}=\kappa\Delta\theta^{\nu}+S, \\
u=M^{\nu}[\theta^{\nu}],\theta(0,x)=\theta_0(x)
\end{array}\right.
\end{align}
via a Fourier multiplier operator $M^{\nu}$ which relates $u^{\nu}$ and $\theta^{\nu}$. More precisely,
\begin{align*}
u^{\nu}_j=M^{\nu}_j [\theta^{\nu}]=(\widehat{M^{\nu}_j}\hat\theta^{\nu})^\vee
\end{align*}
for $j\in\{1,2,3\}$. The explicit expression for the components of $\widehat M^{\nu}$ as functions of the Fourier variable $k=(k_1,k_2,k_3)\in\Z^3$ are obtained from the constitutive law \eqref{1.1-8} to give
\begin{align}
\widehat M^{\nu}_1(k)&=[k_2k_3|k|^2-k_1k_3(k_2^2+\nu|k|^4)]D(k)^{-1},\label{1.2}\\
\widehat M^{\nu}_2(k)&=[-k_1k_3|k|^2-k_2k_3(k_2^2+\nu|k|^4)]D(k)^{-1},\label{1.3}\\
\widehat M^{\nu}_3(k)&=[(k_1^2+k_2^2)(k_2^2+\nu|k|^4)]D(k)^{-1},\label{1.4}
\end{align}
where
\begin{align}\label{1.5}
D(k)=|k|^2k_3^2+(k_2^2+\nu|k|^4)^2.
\end{align}
Here $\kappa>0$ and $\nu\ge0$ are some diffusive constants, $\theta_0$ is the initial condition, and $S=S(x)$ is a given smooth function that represents the forcing of the system. Furthermore, we restrict the system \eqref{1.1} to the function spaces where all functions (including the forcing $S$ and initial data $\theta_0$) have zero mean with respect to $x_3$ (refer to section 4 of \cite{FV11a} for further discussion of this restriction). 

Details of the singular behaviour of the Fourier multiplier symbols for the operator $M^0$ in certain regions of Fourier space are given in \cite{FV11a}. More general issues concerning the ill-posedness and well-posedness of the unforced MG$^0$ equation can be found in \cite{FRV12}, \cite{FV11a}, \cite{FV11b}. In particular, it is to be noted that the MG$^0$ with $\kappa>0$ is the so-called critical MG equation in the sense of the delicate balance between the nonlinear term and the dissipative term. Various critical active scalar equations such as surface quasi-geostrophic equation (SQG) have received considerable attention in the past decade because of the challenging nature of this delicate balance, \cite{CV10}, \cite{CTV14b}, \cite{CV12}, \cite{CW09}, \cite{D10}, \cite{FPV09}, \cite{KN09}, \cite{KNV07}. On the other hand, as we discussed in \cite{FS15}, the MG$^\nu$ equation with $\nu>0$, where the symbol $\widehat{M}^\nu$ decays like $k^{-2}$, is a case where the dissipative term dominates the nonlinear term. In \cite{FS15} it is shown that in the case of the MG$^\nu$ equation with $\nu>0$ and $\kappa>0$, even for singular initial data, the global solution is instantaneously $C^\infty$-smoothed and satisfied classically for all $t>0$. With this dichotomy in mind, we seek to determine the long time behaviour of the forced critical MG$^0$ equation through the ``vanishing viscosity'' limit of the MG$^\nu$ equation. 

\subsection{The MG equations in the class of drift-diffusion equations}

We will now show that the MG$^\nu$, $\nu\ge0$, equations satisfy the conditions of the general class of drift diffusion equations given by (1.11)-(1.14). We write 
\begin{align}
u^{\nu}_j =M^{\nu}_j[\theta^{\nu}]=\partial_{i} T_{ij}^{\nu},\label{1.6}
\end{align}
where we have denoted 
\begin{align}
\mbox{$T_{ij}^{\nu}=-\partial_{i}(-\Delta)^{-1}M^{\nu}_j$ for $\nu\ge0$.}\label{1.7}
\end{align}

\noindent In order to show that conditions (1.11)-(1.14) are satisfied, we need the following lemmas for $\{T^{\nu}_{ij}\}_{\nu\ge0}$:

\begin{lemma}\label{bound on operator}
Let $T_{ij}^{\nu}, T^0_{ij}$ be as defined in \eqref{1.6}-\eqref{1.7} in terms of $M^{\nu}$ and $M^0$. There are constants $C_1,C_2>0$ independent of $\nu$ such that, for all $1\le i,j\le 3$, 
\begin{align}\label{4.1}
\sup_{\nu\in(0,1]}\sup_{\{k\in\mathbb{Z}^3:k\neq0\}}|\widehat T^{\nu}_{ij}(k)|\le\sup_{\nu\in(0,1]}\sup_{\{k\in\mathbb{Z}^3:k\neq0\}}\frac{|\widehat M^{\nu}(k)|}{|k|}\le C_1,
\end{align}
\begin{align}\label{4.2}
\sup_{\{k\in\mathbb{Z}^3:k\neq0\}}|\widehat T^0_{ij}(k)|\le\sup_{\{k\in\mathbb{Z}^3:k\neq0\}}\frac{|\widehat M^0(k)|}{|k|}\le C_2.
\end{align}
\end{lemma}

\begin{proof}
The bound \eqref{4.2} follows from the discussion in (\cite{FV11a}, Section 4) and we omit the proof.

To show the bound \eqref{4.1}, we only give the details for $\widehat M^{\nu}_1$ since the cases for $\widehat M^{\nu}_2$ and $\widehat M^{\nu}_3$ are almost identical. 

To prove \eqref{4.1}, we fix $\nu\in(0,1]$ and consider the following cases:

\noindent{\bf Case 1: $|k|>\nu^{-\frac{1}{2}}$.} Then for each $k\in\mathbb{Z}^3/\{k=0\}$, 
\begin{align*}
\frac{|\widehat M_1^{\nu}(k)|}{|k|}=\frac{|k_2k_3|k|^2-k_1k_3(k_2^2+\nu|k|^4)|}{|k|(|k|^2k_3^2+(k_2^2+\nu|k|^4)^2)}.
\end{align*}
Since $k\neq0$, so $|k|\ge|k_j|\ge1$ for $j=1,2,3$, in particular $|k|^{-1}<\nu^\frac{1}{2}$. Hence we obtain
\begin{align*}
\frac{|\widehat M_1^{\nu}(k)|}{|k|}&\le\frac{|k_2k_3||k|^2}{|k|^3k_3^2}+\frac{|k_1k_3|k_2^2}{|k|^3k_3^2}+\frac{\nu|k_1k_3||k|^4}{\nu^2|k|^8}\\
&\le\frac{1}{|k_3|}+\frac{1}{|k_3|}+\frac{1}{\nu|k|^2}\\
&\le 2+\frac{\nu}{\nu}=3.
\end{align*}

\noindent{\bf Case 2: $|k|\le\nu^{-\frac{1}{2}}$.} Then for each $k\in\mathbb{Z}^3/\{k=0\}$,
\begin{align*}
\frac{|\widehat M_1^{\nu}(k)|}{|k|}&\le\frac{|k_2k_3||k|^2}{|k|^3k_3^2}+\frac{|k_1k_3|k_2^2}{|k|^3k_3^2}+\frac{\nu|k_1k_3||k|^4}{|k|^3k_3^2}\\
&\le\frac{1}{|k_3|}+\frac{1}{|k_3|}+\frac{\nu|k|^2}{|k_3|}\\
&\le2+\nu\cdot(\nu^{-\frac{1}{2}})^2=3.
\end{align*}
Combining two cases, we have
\begin{align*}
\sup_{\nu\in(0,1]}\sup_{\{k\in\mathbb{Z}^3:k\neq0\}}\frac{|\widehat M_1^{\nu}(k)|}{|k|}\le3,
\end{align*}
and hence \eqref{4.1} holds for some $C_1>0$ independent of $\nu$.
\end{proof}
\begin{lemma}\label{convergence of operator}
For each $L>0$, 
\begin{align}\label{4.3}
\lim_{\nu\rightarrow0}\sup_{\{k\in\mathbb{Z}^3:k\neq0,|k|\le L\}}\frac{|\widehat M^{\nu}(k)-\widehat M^0(k)|}{|k|}=0
\end{align}
\end{lemma}
\begin{proof}
Again we only give the details for $\widehat M^{\nu}_1$. We fix $L>0$, then for each $k\in\mathbb{Z}^3\backslash(\{k=0\}$ with $|k|\le L$, we have
\begin{align*}
&\frac{|\widehat M_1^{\nu}(k)-\widehat M^0_1(k)|}{|k|}\\
&=\frac{|-\nu k_1k_3^3|k|^6+\nu k_1 k_2^4k_3|k|^4-\nu^2 k_2k_3|k|^{10}+\nu^2 k_1k_2^2k_3|k|^8-2\nu k_2^3 k_3|k|^6|}{(|k|^2k_3^2+\nu^2|k|^8+2\nu|k|^4k_2^2+k_2^4)(k_3^2|k|^2+k_2^4)|k|}.\\
&\le\frac{\nu|k_1||k_3|^3|k|^6}{|k|^{5}k_3^4}+\frac{\nu|k_1|k_2^4|k_3||k|^4}{|k|^{5}k_3^4}+\frac{\nu^2|k_2||k_3||k|^{10}}{|k|^{5}k_3^4}+\frac{\nu^2|k_1|k_2^2|k_3|k|^8}{|k|^{5}k_3^4}+\frac{2\nu|k_2|^3|k_3||k|^6}{|k|^{5}k_3^4}\\
&\le\nu L^{10}+\nu L^{10}+\nu^2 L^{12}+\nu^2 L^{12}+2\nu L^{10}.
\end{align*}
Hence 
\begin{align*}
\lim_{\nu\rightarrow0}\sup_{\{k\in\mathbb{Z}^3:k\neq0,|k|\le L\}}\frac{|\widehat M_1^{\nu}(k)-\widehat M^0_1(k)|}{|k|}=0.
\end{align*}
\end{proof}

\begin{lemma}\label{convergence of operator 2}
Let $g$ be a function such that $\|\nabla g\|_{L^2}<\infty$. Then we have 
\begin{align}\label{4.4}
\lim_{\nu\rightarrow0}\sum_{k\in\mathbb{Z}^3:k\neq0}\frac{|\widehat M^{\nu}(k)-\widehat M^0(k)|^2|\widehat{\nabla g}(k)|^2}{|k|^2}=0
\end{align}
\end{lemma}

\begin{proof}
Fix $g$ with $\|\nabla g\|_{L^2}<\infty$ and let $\varepsilon>0$ be given. Then $\dis\sum_{k\in\mathbb{Z}^3}|\widehat{\nabla g}(k)|^2<\infty$, so there exists $L=L(\varepsilon)>0$ such that $\dis\sum_{k\in\mathbb{Z}^3, |k|>L}|\widehat{\nabla g}(k)|^2<\varepsilon$. Hence
\begin{align}
&\sum_{k\in\mathbb{Z}^3:k\neq0}\frac{|\widehat M^{\nu}(k)-\widehat M^0(k)|^2|\widehat{\nabla g}(k)|^2}{|k|^2}\notag\\
&=\sum_{k\in\mathbb{Z}^3:k\neq0,|k|\le L}\frac{|\widehat M^{\nu}(k)-\widehat M^0(k)|^2|\widehat{\nabla g}(k)|^2}{|k|^2}+\sum_{k\in\mathbb{Z}^3:k\neq0,|k|>L}\frac{(|\widehat M^{\nu}(k)|^2+|\widehat M^0(k)|^2)|\widehat{\nabla g}(k)|^2}{|k|^2}\notag\\
&\le\left(\sup_{\{k\in\mathbb{Z}^3:k\neq0,|k|\le L\}}\frac{|\widehat M^{\nu}(k)-\widehat M^0(k)|}{|k|}\right)^2\|\nabla g\|_{L^2}^2+(C_1+C_2)\varepsilon.\label{4.5}
\end{align}
Using Lemma~\ref{convergence of operator} and taking $\nu\rightarrow0$ on \eqref{4.5},
\begin{align*}
\lim_{\nu\rightarrow0}\sum_{k\in\mathbb{Z}^3:k\neq0}\frac{|\widehat M^{\nu}(k)-\widehat M^0(k)|^2|\widehat{\nabla g}(k)|^2}{|k|^2}\le (C_1+C_2)\varepsilon.
\end{align*}
Since $\varepsilon>0$ is arbitrary, \eqref{4.4} follows.
\end{proof}

In view of Lemma~\ref{bound on operator}--\ref{convergence of operator 2}, the sequence of operators $\{T^{\nu}_{ij}\}_{\nu\ge0}$ given by \eqref{1.6}-\eqref{1.7} satisfy the conditions (1.11) and (1.13)-(1.14). Moreover, following the discussion given in \cite{FV11a} pp. 298--299, $\{T^{\nu}_{ij}\}_{\nu\ge0}$ also satisfy (1.12). By Theorem~\ref{existence of smooth solutions sub Thm} and Theorem~\ref{uniform bound on solution Thm}, there exists classical solutions $\theta^{\nu}(t,x)\in C^{\infty}((0,\infty)\times \mathbb{T}^3)$ of \eqref{1.1}-\eqref{1.5}, evolving from $\theta_0$ which satisfy the uniform bounds \eqref{2.1}. The abstract Theorem~\ref{Existence of smooth solutions Thm} and Theorem~\ref{Convergence of solutions as nu goes to 0} may therefore be applied to the MG equations in order to obtain the convergence of smooth solutions, and hence we have proven:

\begin{theorem}\label{existence MG}
Let $\theta_0\in L^2$, $S\in C^{\infty}$ and $\kappa>0$ be given. There exists a classical solution $\theta^{\nu}(t,x)\in C^{\infty}((0,\infty)\times \mathbb{T}^3)$ of \eqref{1.1}-\eqref{1.5}, evolving from $\theta_0$ for all $\nu\ge0$.
\end{theorem}

\begin{theorem}\label{convergence MG}
Let $\theta_0\in L^2$, $S\in C^{\infty}$ and $\kappa>0$ be given. Then if $\theta^{\nu},\theta$ are $C^\infty$ smooth classical solutions of \eqref{1.1}-\eqref{1.5} for $\nu>0$ and $\nu=0$ respectively with initial data $\theta_0$, then given $\tau>0$, for all $s\ge0$, we have
\begin{align*}
\lim_{\nu\rightarrow0} \|(\theta^{\nu}-\theta)(t,\cdot)\|_{ H^s}=0,
\end{align*} 
whenever $t\ge\tau$.
\end{theorem}

\section{The Existence of a Global Attractor}\label{The Existence of a Global Attractor}

With the results of Theorems~\ref{existence MG} and \ref{convergence MG} in place, we define a weak solution to the MG$^0$ equation which we call a ``{\it vanishing viscosity}''{\it solution}. We use this concept to prove the existence of a compact global attractor in $L^2(\mathbb{T}^3)$ for the MG$^\nu$ equations \eqref{1.1}-\eqref{1.5} including the critical equation where $\nu=0$. We further obtain the upper semicontinuity of the global attractor as $\nu$ vanishes. First, we define a class of solutions to \eqref{0.1} as follows.

\begin{definition}
A {\it weak solution} to \eqref{1.1}-\eqref{1.5} with $\nu=0$ is a function $\theta\in C_w([0,T];L^2(\mathbb{T}^3))$ with zero spatial mean that satisfies \eqref{1.1} in a distributional sense. That is, for any $\phi\in C_0^\infty((0,T)\times\mathbb{T}^3)$,
\begin{align*}
-\int_0^T\langle\theta,\phi_t\rangle dt-\int_0^T\langle u\theta, \nabla\phi\rangle dt+\kappa\int_0^T\langle\nabla\theta,\nabla\phi\rangle dt=\langle\theta_0,\phi(0,x)\rangle+\int_0^T\langle S,\phi\rangle dt,
\end{align*}
where $u=u\Big|_{\nu=0}$. A weak solution $\theta(t)$ to \eqref{1.1} on $[0,T]$ with $\nu=0$ is called a ``{\it vanishing viscosity}''{\it solution} if there exist sequences $\nu_n\rightarrow0$ and $\{\theta^{\nu_n}\}$ such that $\{\theta^{\nu_n}\}$ are smooth solutions to \eqref{1.1} as given by Theorem~\ref{existence of smooth solutions sub Thm} and $\theta^{\nu_n}\rightarrow\theta$ in $C_w([0,T];L^2)$ as $\nu_n\rightarrow0$. 
\end{definition}
\begin{remark}
By Theorem~\ref{existence MG} and Theorem~\ref{convergence MG}, for any initial data $\theta_0\in L^2$, there exists a ``vanishing viscosity'' solution $\theta$ of \eqref{0.1} on $[0,\infty)$ with $\theta(0)=\theta_0$.
\end{remark}

We prove that the equation \eqref{1.1} driven by a force $S$ possesses a compact global attractor in $L^2(\mathbb{T}^3)$ which is {\it upper semicontinuous} at $\nu=0$. More precisely, we have

\begin{theorem}\label{existence of attractor}
Assume $S\in C^\infty$. Then the system  \eqref{1.1}-\eqref{1.5} with $\nu=0$ possesses a compact global attractor $\mathcal{A}$ in $L^2(\mathbb{T}^3)$, namely
\begin{align*}
\mbox{$\mathcal{A}=\{\theta_0:\theta_0=\theta(0)$ for some bounded complete ``vanishing viscosity'' solution $\theta(t)\}$.}
\end{align*}
For any bounded set $\mathcal{B}\subset L^2(\mathbb{T}^3)$, and for any $\varepsilon,T>0$, there exists $t_0$ such that for any $t_1>t_0$, every ``vanishing viscosity'' solution $\theta(t)$ with $\theta(0)\in\mathcal{B}$ satisfies 
\begin{align*}
\|\theta(t)-x(t)\|_{L^2}<\varepsilon, \forall t\in[t_1,t_1+T],
\end{align*}
for some complete trajectory $x(t)$ on the global attractor $(x(t)\in\mathcal{A}, \forall t\in(-\infty,\infty))$. Furthermore, for $\nu\in[0,1]$, there exists a compact global attractor $\mathcal{A}^{\nu}\subset L^2$ for \eqref{1.1} such that $\mathcal{A}^{0}=\mathcal{A}$ and $\mathcal{A}^{\nu}$ is {\it upper semicontinuous} at $\nu=0$, which means that
\begin{align}\label{uc_0}
\mbox{$\sup_{\phi\in\mathcal{A}^{\nu}}\inf_{\psi\in\mathcal{A}}\|\phi-\psi\|_{L^2}\rightarrow0$ as $\nu\rightarrow0$.}
\end{align}
\end{theorem}

Before we give the proof of Theorem~\ref{existence of attractor}, we state the following proposition. It gives an energy equality which is important in obtaining an absorbing ball for \eqref{1.1} (see Remark~\ref{absorbing ball remark} below). 
\begin{proposition}[The energy equality]\label{energy equality proposition}
Let $\theta(t)$ be a ``vanishing viscosity'' solution of \eqref{1.1} on $[0,\infty)$ with $\theta(0)\in L^2$. Then $\theta(t)$ satisfies the following energy equality:
\begin{align}\label{6.7}
\frac{1}{2}\|\theta(t)\|_{L^2}^2+\kappa\int_{t_0}^t\|\nabla\theta(s,\cdot)\|^2_{L^2}ds=\frac{1}{2}\|\theta(t_0)\|^2_{L^2}+\int_{t_0}^t\int_{\mathbb{T}^3}S\theta dxds,
\end{align}
for all $0\le t_0\le t$.
\end{proposition}
\begin{proof}
It suffices to show that the flux term $\dis\int_{t_0}^t\int_{\mathbb{T}^3}u\theta\cdot\nabla\theta dxds$ equals zero; for which a proof can be found in \cite{CD14} and we omit the details. The proof uses techniques of Littlewood-Paley decomposition. It is analogous to the proof given in \cite{CCF08} that the energy flux is zero for weak solutions of the three dimensional Euler equation that are smoother than Onsager critical.
\end{proof}
\begin{remark}\label{strongly continous}
Based on the equality \eqref{6.7}, we can see that every ``vanishing viscosity'' solution to \eqref{1.1} is strongly continuous in $t$.
\end{remark}
\begin{remark}\label{absorbing ball remark}
Moreover, in view of \eqref{6.7}, there exists an absorbing ball $\mathcal{Y}$ for \eqref{1.1} given by
\begin{align}\label{absorbing ball}
\mathcal{Y}=\{\theta\in L^2: \|\theta\|_{L^2}\le R\},
\end{align}
where $R$ is any number larger than $\kappa^{-1}\|S\|_{H^{-1}(\mathbb{T}^3)}$. Then for any bounded set $\mathcal{B}\subset L^2$, there exists a time $t_0$ such that
\begin{align*}
\theta(t)\in\mathcal{Y},\qquad\forall t\ge t_0,
\end{align*}
for every ``vanishing viscosity'' solution $\theta(t)$ with the initial data $\theta(0)\in\mathcal{B}$.
\end{remark}
\noindent To facilitate the proof of Theorem~\ref{existence of attractor}, we introduce the following notions:
\begin{itemize}
\item[$\cdot$] We denote the strong and weak distances on $L^2(\mathbb{T}^3)$ respectively by
\begin{align*}
d_s(\phi,\psi)=\|\phi-\psi\|_{L^2};\qquad d_w(\phi,\psi)=\sum_{k\in\mathbb{Z}^3}\frac{1}{2^{|k|}}\frac{|\hat\phi_k-\hat\psi_k|}{1+|\hat\phi_k-\hat\psi_k|},
\end{align*}
where $\hat\phi_k$ and $\hat\psi_k$ are the Fourier coefficients of $\phi$ and $\psi$.
\item[$\cdot$] We let 
\begin{align*}
\mathcal{T}=\{\mbox{$I: I=[T,\infty)\subset\R$, or $I=(-\infty,\infty)$}\},
\end{align*}
and for each $I\subset\mathcal{T}$ , let $\mathcal{F}(I)$ denote the set of all $\mathcal{Y}$-valued functions on $I$ (here $\mathcal{Y}$ is the absorbing ball given by \eqref{absorbing ball} in Remark~\ref{absorbing ball remark}). 
\item[$\cdot$] We define $\pi^{\nu}: L^2\rightarrow L^2$ as the map $\pi^{\nu}\theta_0=\theta^{\nu}$, where $\theta^\nu$ is the solution to \eqref{1.1}-\eqref{1.5} given by Theorem~\ref{existence MG}. 
\end{itemize}

We first prove the following lemma which gives the continuity of $\pi^{\nu}(t)\theta_0$ in $\nu$ for a given $\theta_0$. More precisely, we have:
\begin{lemma}\label{continuity in nu lemma}
For $t>0$, $\pi^{\nu}(t)\theta_0$ is continuous in $\nu$, uniformly for $\theta_0$ in compact subsets of $L^2$.
\end{lemma}
\begin{proof}
We let $\mathcal{K}$ be a compact subset of $L^2$ and fix $\theta_0\in\mathcal{K}$. Let $\nu_1,\nu_2\in[0,1]$, then for each $t>0$, we have
\begin{align}\label{uc1}
\frac{1}{2}\frac{d}{dt}\|(\theta^{\nu_1}-\theta^{\nu_2})(t,\cdot)\|_{L^2}^2+\frac{\kappa}{2}\|\nabla(\theta^{\nu_1}-\theta^{\nu_2})(t,\cdot)\|_{L^2}^2=-\int(u^{\nu_1}-u^{\nu_2})\cdot\nabla\theta^{\nu_2}\cdot(\theta^{\nu_1}-\theta^{\nu_2})(t,x)dx,
\end{align}
where $\theta^{\nu_i}(0)=\theta_0$ for $i=1,2$. We follow the argument given in the proof of Theorem~\ref{Convergence of solutions as nu goes to 0} to obtain
\begin{align}
&\left|-\int(u^{\nu}-u)\cdot\nabla\theta^{\nu_2}\cdot(\theta^{\nu_1}-\theta^{\nu_2})(t,x)dx\right|\notag\\
&\le\frac{\kappa}{4C^2}\|(u^{\nu_1}-u^{\nu_2})(t,\cdot)\|^2_{L^2}+\frac{4C^2}{\kappa}\|(\theta^{\nu_1}-\theta^{\nu_2})(t,\cdot)\|_{L^2}^2\|\nabla\theta^{\nu_2}(t,\cdot)\|_{L^\infty}^2.\label{uc2}
\end{align}
Using the bound \eqref{2.1} for $\theta^{\nu_2}$, the second term on the right side of \eqref{uc2} is bounded by \newline $C\|(\theta^{\nu_1}-\theta^{\nu_2})(t,\cdot)\|_{L^2}^2$, for some constant $C>0$ independent of $\theta_0$ but depends on the compact set $\mathcal{K}$ and $t$. On the other hand, the term $\|(u^{\nu_1}-u^{\nu_2})(t,\cdot)\|^2_{L^2}$ can be bounded as follows.
\begin{align*}
\|(u^{\nu_1}-u^{\nu_2})(t,\cdot)\|^2_{L^2}&=\sum_{k\in\mathbb{Z}^3}|(\widehat{\partial_{i} T_{ij}^{\nu_1}}\widehat{\theta^{\nu_1}}-\widehat{\partial_{i} T_{ij}^{\nu_2}}\widehat{\theta^{\nu_2}})(k)|^2\notag\\
&\le\sum_{k\in\mathbb{Z}^3}|\widehat{T_{ij}^{\nu_1}}(k)|^2|\widehat{\nabla(\theta^{\nu_1}-\theta^{\nu_2})}(k)|^2+\sum_{k\in\mathbb{Z}^3}|\widehat{T^{\nu_1}_{ij}}(k)-\widehat{T^{\nu_2}_{ij}}(k)|^2|\widehat{\nabla\theta^{\nu_2}}(k)|^2\\
&\le C^2\|\nabla(\theta^{\nu_1}-\theta^{\nu_2})(t,\cdot)\|_{L^2}^2+\sum_{k\in\mathbb{Z}^3}|\widehat{T^{\nu_1}_{ij}}(k)-\widehat{T^{\nu_2}_{ij}}(k)|^2|\widehat{\nabla\theta^{\nu_2}}(k)|^2,
\end{align*}
where $T^{\nu_1}_{ij},T^{\nu_2}_{ij}$ are defined in \eqref{1.7}. Using the bound \eqref{2.1} for $\theta^{\nu_2}$ again and applying the similar argument given in the proof of Lemma~\ref{bound on operator} and Lemma~\ref{convergence of operator}, there exists constant $C=C(\nu_1,\nu_2,\mathcal{K},t)>0$ such that
\begin{align*}
\sum_{k\in\mathbb{Z}^3}|\widehat{T^{\nu_1}_{ij}}(k)-\widehat{T^{\nu_2}_{ij}}(k)|^2|\widehat{\nabla\theta^{\nu_2}}(k)|^2\le C|\nu_1-\nu_2|.
\end{align*}
Hence we conclude from \eqref{uc1} that
\begin{align*}
\frac{1}{2}\frac{d}{dt}\|(\theta^{\nu_1}-\theta^{\nu_2})(t,\cdot)\|_{L^2}^2\le C\left[|\nu_1-\nu_2|+\|(\theta^{\nu_1}-\theta^{\nu_2})(t,\cdot)\|_{L^2}^2\right].
\end{align*}
Integrating the above over $t$ and using Gr\"{o}nwall's inequality, it further implies
\begin{align}\label{continuity in nu} 
\|(\pi^{\nu_1}(t)\theta_0-\pi^{\nu_2}(t)\theta_0)\|_{L^2}^2=\|(\theta^{\nu_1}-\theta^{\nu_2})(t,\cdot)\|_{L^2}^2&\le C e^{Ct}\int_0^t|\nu_1-\nu_2|ds\notag\\
&\le C e^{Ct}t|\nu_1-\nu_2|,
\end{align}
hence we prove the continuity of $\pi^{\nu}$ in $\nu$ uniformly for $\theta_0$ in $\mathcal{K}$.
\end{proof}

Next, the following lemma shows that the {\it weak} upper semicontinuity implies the {\it strong} upper semicontinuity.
\begin{lemma}\label{weak implies strong}
Let $\phi^{\nu_j}\in\mathcal{A}^{\nu_j}$ and $\psi_j\in\mathcal{A}$ be such that
\begin{align*}
\lim_{j\rightarrow\infty} d_w(\phi^{\nu_j},\psi_j)=0,
\end{align*}
for some sequence $\nu_j\rightarrow0$. Then
\begin{align*}
\lim_{j\rightarrow0}\|\phi^{\nu_j}-\psi_j\|_{L^2}=0.
\end{align*}
\end{lemma}
\begin{proof}
Assume the conclusion of the lemma does not hold. Then passing to a subsequence and dropping a subindex, we can assume that
\begin{align}\label{contradiction}
\liminf_{j\rightarrow0}\|\phi^{\nu_j}-\psi_j\|_{L^2}>0.
\end{align}
There are solutions $\theta^{\nu_j}_j(t,\cdot),\theta^{0}_j(t,\cdot)$ (with $\nu=\nu_j$ and $\nu=0$ respectively) which are complete bounded in $L^2$ such that $\theta^{\nu_j}_j(\cdot,1)=\phi^{\nu_j}$ and $\theta^0_j(\cdot,1)=\psi_j$.

Due to the energy equality \eqref{6.7} and the fact that the radius of the absorbing ball $\mathcal{Y}$ does not depend on $\nu$, there exists a constant $C>0$ independent of $\nu$ such that
\begin{align*}
\int_0^1\|\nabla\theta^{\nu_j}_j(t,\cdot)\|_{L^2}^2dt\le C,\qquad\int_0^1\|\nabla\theta^0_j(t,\cdot)\|_{L^2}^2dt\le C,
\end{align*}
for all $j$. This implies there  exists a sequence $t_j\in[0,1]$ such that
\begin{align}\label{bound on nablatheta}
\|\nabla\theta^{\nu_j}_j(\cdot,t_j)\|_{L^2}^2\le C,\qquad\|\nabla\theta^0_j(\cdot,t_j)\|_{L^2}^2\le C.
\end{align}
Since the interval $[0,1]$ is compact, passing to a subsequence and dropping a subindex, we can assume that $t_j\rightarrow t^*$ for some $t^*\in[0,1]$. And using \eqref{bound on nablatheta}, by compactness, we can pass to another subsequence and drop a subindex to obtain
\begin{align*}
\mbox{$\theta^{\nu_j}_j(\cdot,t_j)\rightarrow\bar{\theta}_0(\cdot)$, \,$\theta^0_j(\cdot,t_j)\rightarrow\bar{\bar{\theta}}_0(\cdot)$ in $L^2$,}
\end{align*}
for some $\bar{\theta}_0,\bar{\bar{\theta}}_0\in L^2$. 

For $\nu=0$, we consider solutions $\bar{\theta}$ and $\bar{\bar{\theta}}$ with $\bar{\theta}(\cdot,t^*)=\bar{\theta}_0$ and $\bar{\bar{\theta}}(\cdot,t^*)=\bar{\bar{\theta}}_0$. Note that we have $\theta^{\nu_j}(t_j)\rightarrow\bar{\theta}(t^*)$ in $L^2$ as $t_j\rightarrow t^*$. Following a similar proof of \eqref{continuity in nu}, there exists a constant $C=C(\bar{\theta}_0)>0$ independent of $\nu$ such that
\begin{align*}
\lim_{j\rightarrow\infty}\|\theta^{\nu_j}_j(\cdot,1)-\bar{\theta}(\cdot,1)\|_{L^2}\le\lim_{j\rightarrow\infty}C\int_{t^*}^1\nu_{j}ds=0.
\end{align*}
Similarly, we also have 
\begin{align*}
\lim_{j\rightarrow\infty}\|\theta^{0}_j(\cdot,1)-\bar{\bar{\theta}}(\cdot,1)\|_{L^2}=0.
\end{align*}
So we have shown that there is a sequence $j_n$ such that
\begin{align*}
\lim_{n\rightarrow\infty}\|\phi^{\nu_{j_n}}(\cdot)-\bar{\theta}(\cdot,1)\|_{L^2}=\lim_{n\rightarrow\infty}\psi_{j_n}(\cdot)-\bar{\bar{\theta}}(\cdot,1)\|_{L^2}=0.
\end{align*}
Hence $\bar{\theta}(\cdot,1)=\bar{\bar{\theta}}(\cdot,1)$ and we have $\|\phi^{\nu_{j_n}}-\psi_{j_n}\|_{L^2}\rightarrow0$ as $n\rightarrow\infty$, which contradicts \eqref{contradiction}.
\end{proof}
We are now ready to prove Theorem~\ref{existence of attractor}. 
\begin{proof}[Proof of Theorem~\ref{existence of attractor}] We define 
\begin{align*}
\mbox{$\mathcal{E}[T,\infty)=\{\theta(\cdot):\theta(\cdot)$ is a ``vanishing viscosity'' solution of \eqref{0.1}}\\
\mbox{on $[T,\infty)$ and $\theta\in\mathcal{Y}$ for all $t\in[T,\infty)\}$},
\end{align*}
\begin{align*}
\mbox{$\mathcal{E}(-\infty,\infty)=\{\theta(\cdot):\theta(\cdot)$ is a ``vanishing viscosity'' solution of \eqref{0.1}}\\
\mbox{on $(-\infty,\infty)$ and $\theta\in\mathcal{Y}$ for all $t\in(-\infty,\infty)\}$},
\end{align*}
then $\mathcal{E}$ is an {\it evolutionary system} (see \cite{C09} and \cite{CD14} for the definition), so by Theorem~4.5 in \cite{CD14}, there exists a weak global attractor $\mathcal{A}_w$ to $\mathcal{E}$ with
\begin{align}
\mbox{$\mathcal{A}_w = \{\theta_0 : \theta_0 = \theta(0)$ for some $\theta \in \mathcal{E}((-\infty,\infty))\}$}.
\end{align}
Furthermore, by the Aubin-Lions Lemma (also refer to \cite{CF89} for the case of Navier-Stokes equation), if $\theta_n(t)$ is any sequence of ``vanishing viscosity'' solutions of \eqref{0.1} such that $\theta_n(t)\in\mathcal{Y}$ for all $t\ge t_0$, then there exists a subsequence $\theta_{n_j}$ of $\theta_n$ that converges in $C([t_0,T];\mathcal{Y}_w)$ to some ``vanishing viscosity'' solution $\theta(t)$ (here $\mathcal{Y}_w$ refers to the metric space $(\mathcal{Y}, d_w)$). Applying the arguments given in \cite{CD14}, $\mathcal{E}$ satisfies all the following properties: 
\begin{itemize}
\item[A1] $\mathcal{E}([0, \infty))$ is a compact set in $C([0, \infty); \mathcal{Y}_w)$ ($\mathcal{Y}_w$ is endowed with the weak topology induced by $d_w$);
\item[A2] for any $\varepsilon>0$, there exists $\delta>0$ such that for every $\theta\in\mathcal{E}([0,\infty))$ and $t >0$, $$\|\theta(t)\|_{L^2}\le \|\theta(t)\|_{L^2} + \varepsilon,$$ for $t_0$ a.e. in $(t-\delta,t)\cap[0,\infty)$;
\item[A3] if $\theta_n\in\mathcal{E}([0,\infty))$ and $\theta_n\rightarrow\theta\in\mathcal{E}([0,\infty))$ in $C([0, \infty); \mathcal{Y}_w)$ for some $T>0$, then $\theta_n(t)\rightarrow\theta(t)$ strongly a.e. in $[0,T]$.
\end{itemize}
Therefore, together with Remark~\ref{absorbing ball remark}, Theorem~4.5 in \cite{CD14} can then be applied again to our evolutionary system $\mathcal{E}$, which implies that
\begin{itemize}
\item the strong global attractor $\mathcal{A}_s$ exists, it is strongly compact and $\mathcal{A}:=\mathcal{A}_s = \mathcal{A}_w$; and 
\item for any bounded set $\mathcal{B}\subset L^2(\mathbb{T}^3)$, and for any $\varepsilon,T>0$, there exists $t_0$ such that for any $t_1>t_0$, every ``vanishing viscosity'' solution $\theta(t)$ $$\|\theta(t)-x(t)\|_{L^2}<\varepsilon, \forall t\in[t_1,t_1+T],$$
for some complete trajectory $x(t)$ on the global attractor $(x(t)\in\mathcal{A}, \forall t\in(-\infty,\infty))$.
\end{itemize}
Finally, to prove that \eqref{uc_0} holds, we recall that $\pi^{\nu}: L^2\rightarrow L^2$ is the map $\pi^{\nu}\theta_0=\theta^{\nu}$, where $\theta^\nu$ is the solution to \eqref{1.1}-\eqref{1.5} given by Theorem~\ref{existence MG}. Then for each $\nu\in(0,1]$, $\pi^{\nu}$ is a semigroup $\{\pi^{\nu}(t)\}_{t\ge0}$ on $L^2$. And using the above argument, there exists a compact global attractor $\mathcal{A}^{\nu}\subset L^2$ for $\pi^{\nu}$ given by
\begin{align*}
\mbox{$\mathcal{A}^{\nu}= \{\theta_0 : \theta_0 = \theta^\nu(0)$, where $\theta^\nu$ is a solution of \eqref{1.1}-\eqref{1.5}}\\
\mbox{defined on $(-\infty,\infty)$ and $\theta^\nu\in\mathcal{Y}$ for all $t\in(-\infty,\infty)\}$}.
\end{align*}
Moreover, for $\nu=0$, $\mathcal{A}^0=\mathcal{A}$ is the global attractor for $\pi^{0}(\cdot)$ satisfying
\begin{itemize}
\item[$\cdot$] $\pi^{0}(t)\mathcal{A}=\mathcal{A}$ for all $t\in\R$; 
\item[$\cdot$] for any bounded set $\mathcal{B}$, $\sup_{\phi\in\pi^{0}(t)\mathcal{B}}\inf_{\psi\in\mathcal{A}}\|\phi-\psi\|_{L^2}\rightarrow0$ as $t\rightarrow0$.
\end{itemize}
We claim that the following conditions hold:
\begin{itemize}
\item[L1] $\pi_\nu(\cdot)$ has a global attractor $\mathcal{A}^{\nu}$ for every $\nu\in[0,1]$ in the {\it weak-$L^2$ sense}, which means that
\begin{itemize}
\item[$\cdot$] $\pi^{0}(t)\mathcal{A^\nu}=\mathcal{A^\nu}$ for all $t\in\R$; 
\item[$\cdot$] for any bounded set $\mathcal{B}$, $\sup_{\phi\in\pi^{0}(t)\mathcal{B}}\inf_{\psi\in\mathcal{A}}d_{w}(\phi,\psi)\rightarrow0$ as $t\rightarrow0$.
\end{itemize}
\item[L2] there is a {\it compact subset} $\mathcal{K}$ of $L^2$ in the weak topology induced by $d_w$ such that $\mathcal{A}^{\nu}\subset \mathcal{K}$ for every $\nu\in[0,1]$.
\item[L3] for $t>0$, $\pi^{\nu}(t)\theta_0$ is continuous in $\nu$, uniformly for $\theta_0$ in compact subsets of $L^2$.
\end{itemize}
Notice that Lemma~\ref{continuity in nu lemma} implies L3, so we only have to show L1 and L2:
\begin{itemize}
\item[$\cdot$] To show L1, we note that the absorbing ball $\mathcal{Y}$ as given by \eqref{absorbing ball} has radius which is independent of $\nu$, hence $\pi_\nu(\cdot)$ has a global attractor $\mathcal{A}^{\nu}$ for every $\nu\in[0,1]$ satisfying L1. 
\item[$\cdot$] To show L2, using A1, we have that $\mathcal{E}([0, \infty))$ is a compact set in $C([0, \infty); \mathcal{Y}_w)$, where $\mathcal{Y}_w$ refers to the metric space $(\mathcal{Y}, d_w)$ endowed with the weak topology induced by $d_w$. Hence we take $\mathcal{K}=\mathcal{Y}_w$ and $\mathcal{A}^{\nu}\subset \mathcal{K}$ for every $\nu\in[0,1]$.
\end{itemize}
In view of L1 to L3, the result from \cite{HOR15} implies the {\it weak} upper semicontinuity, namely
\begin{align}\label{weak upper continuity}
\mbox{$\sup_{\phi\in\mathcal{A}^\nu}\inf_{\psi\in\mathcal{A}}d_w(\phi,\psi)\rightarrow0$ as $\nu\rightarrow0$.}
\end{align}
Hence using Lemma~\ref{weak implies strong}, \eqref{weak upper continuity} further implies the {\it strong} upper semicontinuity given by \eqref{uc_0}. This completes the proof of Theorem~\ref{existence of attractor}.
\end{proof}


\subsection*{Acknowledgment} We thank Vlad Vicol for his very helpful advice. We also thank the referees for their most valuable comments. A.S. is partially supported by Hong Kong Early Career Scheme (ECS) grant project number 28300016. S.F. is partially supported by NSF grant DMS-1613135.

\end{document}